\documentclass{article}

\usepackage{arxiv}
\usepackage{array}
\usepackage[utf8]{inputenc} 
\usepackage[T1]{fontenc}    
\usepackage{hyperref}       
\usepackage{url}            
\usepackage{booktabs}       
\usepackage{amsfonts}       
\usepackage{nicefrac}       
\usepackage{microtype}      
\usepackage{lipsum}
\usepackage{graphicx}

\usepackage{amssymb}
\usepackage{color}
\usepackage[T1]{fontenc}
\usepackage{amsthm}
\usepackage{amsmath,amsfonts}
\usepackage{enumitem}
\usepackage{appendix}
\usepackage{algorithm}
\usepackage{algcompatible}

\usepackage{hyperref}

\newtheorem{proposition}{Proposition}

\title{Transient Instability and Patterns of Reactivity in Diffusive-Chemotaxis  Soil Carbon Dynamics}

\author{Fasma Diele \\
  Istituto per le Applicazioni del Calcolo \lq \lq M. Picone\rq \rq\\
  National Research Council (CNR)\\
  via G. Amendola 122/D, Bari, Italy\\
  \texttt{fasma.diele@cnr.it}
  \And
  Andrew L. Krause \\
  Mathematical Sciences Department \\
  Durham University\\
  Upper Mountjoy Campus, Stockton Rd, Durham, United Kingdom\\
  \texttt{andrew.krause@durham.ac.uk}
  \And
  Deborah Lacitignola \\
  Dipartimento di Ingegneria Elettrica e dell'Informazione\\
  Universit\'a di Cassino e del Lazio Meridionale\\
  via Di Biasio, Cassino, Italy\\
  \texttt{d.lacitignola@unicas.it}
   \And
 Carmela Marangi \\
  Istituto per le Applicazioni del Calcolo \lq \lq M. Picone\rq \rq\\
  National Research Council (CNR)\\
  via G. Amendola 122/D, Bari, Italy\\
  \texttt{carmela.marangi@cnr.it} \\
  \And
 Angela Monti \\
  Istituto per le Applicazioni del Calcolo \lq \lq M. Picone\rq \rq \\
  National Research Council (CNR)\\
  via G. Amendola 122/D, Bari, Italy\\
  \texttt{angela.monti@cnr.it} 
  \And
  Edgardo Villar-Sep\'ulveda \\
  School of Engineering Mathematics and Technology \\
  University of Bristol \\
  Ada Lovelace Building, Tankard’s Cl, University Walk, Bristol, United Kingdom\\
  \texttt{edgardo.villar-sepulveda@bristol.ac.uk}
}

\newcommand{\jac}{J_{k}}

\begin{document}

\maketitle

\begin{abstract}
We study pattern formation in a chemotaxis model of bacteria and soil carbon dynamics as an example system where transient dynamics can give rise to pattern formation outside of Turing unstable regimes. We use a detailed analysis of the reactivity of the non-spatial and spatial dynamics, stability analyses, and numerical continuation to uncover detailed aspects of this system's pattern-forming potential. In addition to patterning in Turing unstable parameter regimes, reactivity of the spatial system can itself lead to a range of parameters where a spatially uniform state is asymptotically stable, but exhibits transient growth that can induce pattern formation. We show that this occurs in the bistable region of a subcritical Turing bifurcation. Intriguingly, such bistable regions appear in two spatial dimensions, but not in a one-dimensional domain, suggesting important interplays between geometry, transient growth, and the emergence of multistable patterns. We discuss the implications of our analysis for the bacterial soil organic carbon system, as well as for reaction-transport modeling more generally.
\end{abstract}

\section{Introduction}

Pattern formation in diffusive-chemotaxis models has emerged as an important area of study for understanding complex spatial structures in biological, ecological, and chemical systems. Such models also serve as simple examples of pattern formation without the constraints needed in simpler reaction-diffusion frameworks, such as having a short-range activator and a long-range inhibitor. In this paper, we will focus on the analysis of a reaction-diffusion-chemotaxis model of relevance in studying hotspot behaviours in bacterial soil carbon systems, and demonstrate some nontrivial aspects of such systems in the context of transient instability and pattern formation outside of the classical Turing space.

Chemotaxis, the ability of certain bacteria to direct their movement in response to chemical gradients, appears in various biological processes. For instance, in soil ecosystems, motile and chemotactic microorganisms that degrade soil organic carbon (SOC) influence carbon cycling and soil health \cite{vogel2014submicron}. Despite its ecological significance, chemotaxis has largely been neglected in terrestrial carbon cycle models, such as Century \cite{parton1996century} and RothC \cite{coleman1997simulating}, which primarily rely on compartmental systems of spatially implicit differential equations (ordinary or, more recently, of fractional order \cite{bohaienko2023novel}).

This gap was addressed in \cite{hammoudi2018mathematical}, where the authors proposed a novel formulation of the original ODE-based MOMOS model \cite{pansu2010modeling} by incorporating the chemotactic movement of bacteria along with diffusive dynamics. The new model of soil carbon dynamics, which is a reaction-diffusion system with a chemotactic term, accounts for the formation of soil aggregations in the bacterial and microorganism spatial organization (hot spots in soil). This spatial and chemotactic version of MOMOS, supported by validated parameters and experimental data, suggests that accounting for chemotaxis can enhance the predictive capacity of soil carbon models, particularly in understanding soil CO$_2$ emissions and land management strategies \cite{del2005modeling}.


Traditionally, studies on pattern formation in reaction-diffusion and chemotaxis systems have focused on asymptotic stability to explain the emergence of patterns \cite{Maini_1997,Murray_book2,Lacitignola_2015,Giunta2021}. Asymptotic linear stability analysis has been successful in describing {pattern-forming (Turing) instabilities} but fails to capture transient phenomena that may {also give rise to} asymptotically stable patterned states. Transient instability, characterized by significant deviations from {an} equilibrium before eventual stabilization, is increasingly recognized as a potential driver of patterns, particularly in systems exhibiting non-normality \cite{muolo2019patterns,di2018non}. Non-normality refers to the lack of orthogonality in the eigenfunctions of a system, which leads to interactions that can transiently amplify perturbations before eventual decay. Initial amplification is possible in reactive systems, i.e., systems where the numerical abscissa associated with the linearized dynamics, which measures the initial growth rate, is positive \cite{NEUBERT2002}. However, the magnitude of the initial growth rate does not provide information about the maximum amplification that can occur during the transient phase \cite{neubert1997alternatives}. This analysis can be carried out by introducing the concept of "pseudospectra," as developed by Trefethen in \cite{trefethen2005spectra}, who provided a robust framework for analyzing such transient dynamics. 


The topic of transient instabilities was initially explored in the context of pattern formation by \cite{NEUBERT2002}, where reactivity of the (non-spatial) Jacobian matrix was shown to be necessary for diffusion-driven instability in reaction-diffusion systems, and in various related modeling frameworks such as coupled map lattices. Besides this general result, the effects of transient growth of perturbations on the spatial dynamics have been often neglected, partially due to evidence that non-normality-induced patterns are considered rare events and therefore {possibly not} biologically relevant \cite{klika2017significance}.  {An early example of such non-normality-induced pattern formation was made in \cite{ridolfi2011transient}}, and more recently, the phenomenon has been investigated in networked systems \cite{muolo2019patterns} and neural dynamics \cite{di2018non}.  \cite{klika2017significance} showed that the size of a set in the parameter space where transient growth is significant (for  reaction-diffusion models) was limited to small regions near the boundary of the Turing-unstable parameter space.

The implications for chemotaxis-driven patterns {induced by transient growth are} largely unexplored. In this work, inspired by the approach outlined by \cite{klika2017significance}, we analyze the diffusive-chemotaxis MOMOS model for soil carbon dynamics. This model serves as an ideal framework for investigating these effects, given its established utility in modeling SOC dynamics and its potential for incorporating chemotaxis as a driving mechanism for pattern formation \cite{hammoudi2018mathematical, pansu2010modeling}.
%
%
%
 The analysis begins by establishing the set of parameter values corresponding to a stable and reactive equilibrium of the linearized dynamics {(i.e.~parameters where transient patterning instabilities occur, despite the homogeneous equilibrium being stable}). Within this region, we identified the emergence of patterns (referred to as reactive patterns).
For these parameters, we estimated the maximum amplification{, which is an indicator of transient growth rates,} using the indicator proposed by \cite{klika2017significance}. Our findings show that this indicator provides a more accurate lower bound compared to {other measures of transient growth, such as} the Kreiss constant {for the observed amplification of perturbations in our system} \cite{trefethen2005spectra}.

Furthermore, we observed that high maximum amplification, when coupled with a short return time after perturbation, prevents the emergence of stable reactive patterns. To address this, we estimated the return time using a proxy given by the determinant of the linearized Jacobian. Indeed, when the determinant approaches zero, the return time becomes infinitely large, allowing the kinetics to influence the dynamics significantly. Finally, we identified a very small region near the instability boundary where multiple stable reactive patterns emerge{, and investigated these regions using numerical continuation. We find that these stable reactive patterns exist due to the bistability emergent from a nearby subcritical Turing instability, with the size of the bistable region dependent on the geometry. Strikingly, the bistable region does not exist in one spatial dimension, as the criticality of the bifurcation changes between one and two dimensional systems. Overall this leads to an interesting interplay between reactivity and geometry which permits pattern formation outside of the classical Turing unstable regime.}

{Our results have potential biological implications both for the specific applications of the MOMOS model in soil carbon dynamics, and more generally for systems involving cross-diffusion-induced instabilities. The  specific reduced MOMOS model and parameter regime itself was chosen 
to loosely reflect dynamics in the more complex (though spatially-homogeneous) MOMOS models reviewed by \cite{pansu2010modeling}, though we anticipate that the phenomena of transient instability leading to pattern formation is plausible in more complex versions of soil organic carbon dynamics, as well as in other cross-diffusion systems more generally. We also highlight the role of bistability emergent from subcritical Turing instabilities in shaping dynamics outside of the classical Turing space as a general motif which is not as well-studied in biological applications compared to the classical supercritical picture (see, for example, \cite{champneys2021bistability, brennan2025pattern} for recent examples of the role of such phenomena in development, vegetation patterning, and oncology).}

The remainder of this paper is organized as follows: Section \ref{Sec_app_chem_insta} revisits the conditions for asymptotic instability in chemotaxis-diffusion systems and establishes the basic notations used throughout the paper in the context of general chemotaxis models. Section \ref{Sec_characterizing} focuses on characterizing transient instability, exploring the concept of non-normality as a general framework, reactivity as a driver of initial amplification, and the maximum amplification envelope as a characterization of transient dynamics. In Section \ref{Sec_patterns}, we analyze patterns of reactivity in the MOMOS model, identifying regions in the parameter space where transient dynamics lead to the emergence of stable patterns, {which we demonstrate through direct simulation}. {In Section \ref{Sec_continuation}, we give results from numerical continuation and a weakly nonlinear stability analysis to explain the existence of stable patterned states via bistability arising from subcritical Turing bifurcations.} Finally, we summarize the main findings and discuss their implications for soil carbon modeling {and in more general contexts} in Section \ref{Sec_Discussion}.

\section{Asymptotic  instability}
\label{Sec_app_chem_insta}

We consider the general reaction-diffusion-chemotaxis model
 \begin{equation}
\label{eq:general_hu}  
\partial_t u = D_u \Delta u -\beta \, \nabla \cdot( \ell(u) \nabla v) + f(u,v), \quad  \partial_t v = D_v \Delta v + g(u,v),
\end{equation}
with given initial and boundary conditions, $\beta>0$ the chemotactic sensitivity, {$\ell(u)$ the density-dependence of the chemoattraction,} and $D_u, D_v>0$ the diffusion coefficients of each species respectively. {For concreteness, we will consider two-dimensional square domains with both species satisfying periodic boundary conditions unless stated otherwise.} In this section, we recap the main steps that lead to the condition for diffusion-chemotaxis driven instability. While these are standard arguments, and are special cases of more general analyses such as by \cite{ritchie2022turing}, we include them here for completeness and to introduce relevant notations for later. 

Consider $(u_0,v_0)$ a spatially homogeneous equilibrium ($f(u_0,v_0)=g(u_0,v_0) = 0$) which is assumed to be stable in the absence of diffusion and chemotaxis. This assumption requires that the entries of $J_{0}$, the Jacobian matrix evaluated at the steady state,
 \begin{equation}
     \label{eq:jacobian}
     J_{0} = \left [ \begin{array}{cc}
          f_u & f_v \\
          g_u & g_v
     \end{array}
     \right],
\end{equation}  satisfy the conditions 
 \begin{equation}\label{eq:conditions_lin}
f_u + g_v < 0, \quad f_u g_v - f_v g_u > 0.
 \end{equation}
 
 By linearizing the full system \eqref{eq:general_hu} about $(u_0,v_0)$,  we obtain the equation
\begin{equation}
\label{eq:general_full_lin}
   \mathbf{w}_t = L \Delta \mathbf{w} + J_{0} \mathbf{w},
\end{equation}
 for the vector of perturbations  $
     \mathbf{w} = \left [ \begin{array}{c}
           u - u_0 \\
           v - v_0
     \end{array}
     \right ]
$  where $J_{0}$ is given by \eqref{eq:jacobian} and 
    $$
    L = \left [ \begin{array}{cc}
         D_u & -\beta \,\ell(u_0) \\
         0 & D_v
    \end{array}
    \right].
    $$

{We will assume that the domain is sufficiently large, in order to justify the standard linear stability analysis of such systems on $\mathrm{R}^2$ (see, e.g., \cite{Murray_book2}). We take the Fourier transform of $\mathbf{w}$,
  $$  \widetilde{\mathbf{w}}({\boldsymbol\omega},t) = \int_{\mathbb{R}^2} e^{i {\boldsymbol\omega} \cdot \mathbf{x}} \mathbf{w}(\mathbf{x},t) d\mathbf{x},
  $$
and define the matrix
\begin{equation}
    \label{eq:j_matrix}
    \jac = J_{0} - k^2 L = J_{0} - k^2 D - k^2 C,
\end{equation}
where $k = \| {\boldsymbol\omega}\|$ {is the wavenumber (the Euclidean norm of the wavevector $\boldsymbol{\omega}$), and}
$$
    D = \left [ \begin{array}{cc}
         D_u & 0 \\
         0 & D_v
    \end{array}
    \right], \quad C = \left [ \begin{array}{cc}
         0 & -\beta \, \ell(u_0) \\
         0 & 0
    \end{array}
    \right].
$$
Then, we transform \eqref{eq:general_full_lin} in
\begin{equation}  \label{eq:general_linear}
\widetilde{\mathbf{w}}_t
  = \jac \widetilde{\mathbf{w}},
  \end{equation}
and compute the eigenvalues of the matrix $\jac$ {to determine the growth rate of the perturbations, $\mathbf{w}$}. Thus we have
    $$
    \left| \lambda I - J_{0} + L k^2 \right| = 0,
    $$
which gives
 \begin{equation}
 \label{eq:charpol}
    \lambda^2 + \lambda \left( \left(D_u+D_v \right)k^2 - \left(f_u+g_v\right) \right) + h(k^2) = 0,
 \end{equation}
with
   \label{eq:h}
  $$
    h(k^2) =  D_u D_v k^4  - \left(D_u g_v + D_v f_u + \beta g_u \ell(u_0) \right) k^2 + \left(f_u g_v - f_v g_u \right).
    $$
In order to obtain an unstable steady state in presence of diffusion and chemotaxis, we want that the real part of at least one root of the characteristic polynomial \eqref{eq:charpol} is positive for some $k \neq 0$. As a consequence of \eqref{eq:conditions_lin} the term    $\left(D_u + D_v \right) k^2 - \left(f_u + g_v \right) > 0$, {which implies} that imaginary roots of the characteristic polynomial (\ref{eq:charpol}) have negative real part. Focusing on real roots, according to Descartes' rule of signs, we need \(h(k^2) < 0\) for some \(k \neq 0\) to assure the existence of a positive (real) solution of (\ref{eq:charpol}).
Since \( h(k^2) \) is a second-order polynomial in $k^2$ with a positive coefficient for the quadratic term, we require that the discriminant is positive, i.e.,
\begin{equation}\label{eq:hmin}
   \Delta:= \left(D_u g_v + D_v f_u + g_u \beta \,\ell(u_0)\right)^2 - 4 \, D_u \, D_v \, \left(f_u g_v - f_v g_u\right) > 0,
\end{equation}
otherwise, the polynomial would be positive for all values of \( k^2 \).
  Moreover, we should require that there exists at least one positive root to ensure that there is a range of positive values of \( k^2 \) where \( h(k^2) \) assumes negative values. Again, according to 
Descartes' rule of signs, as the existence of a positive root corresponds to a variation of the signs of {the coefficients of} $h(k^2)$, we require that  
\begin{equation}\label{eq:k2minimo}
    \left(D_u g_v + D_v f_u + \beta g_u \, \ell(u_0) \right) >0.
\end{equation}
If both conditions (\ref{eq:hmin}) and (\ref{eq:k2minimo}) and are satisfied then \eqref{eq:charpol} admits a positive solution for a given range of wavenumbers. By solving the equation \(h(k^2) = 0\), and considering that the second inequality in (\ref{eq:conditions_lin}) implies that \(h(k^2)\) has two positive roots, we can determine the range of unstable wavenumbers $k$,
\begin{equation}
\label{eq_range_wave}
    \begin{aligned}
    k_1^2 &:= \frac{A + g_u \beta  \, \ell(u_0) - \sqrt{\Delta}}{2 D_u D_v} < k^2 
    < \frac{A + g_u \beta \,\ell(u_0) +\sqrt{\Delta}}{2 D_u D_v} =: k_2^2 
    \end{aligned}
\end{equation}
where $A=D_u g_v + D_v f_u $.

 Finally, we summarize these {necessary conditions for Turing instability (which become sufficient on sufficiently large domains) as:}
\begin{equation}
    \label{eq:suff_conditions_cd}
    \begin{aligned}
        & f_u + g_v < 0, \quad  f_u g_v - f_v g_u > 0 
        \\
        & D_u g_v + D_v f_u + g_u \beta  \,\ell(u_0) > 2\sqrt{ D_u D_v \left( f_u g_v - f_v g_u \right)} 
        \end{aligned}
\end{equation}
where the first two inequalities ensure stability of $J_{0}$ and the last implies chemotaxis-diffusion instability of $\jac$.
Notice that, for $\beta=0$, we recover the classical and well known conditions for Turing (diffusion-driven) instabilities \cite{Murray_book2}. {The terminology we adopt is that a Turing instability is one where a stable homogeneous equilibrium is driven to instability by the addition of spatial transport terms, leading to a dispersion relation where large and small wavenumbers are stable, but some range of finite wavenumbers are unstable \cite{krause2021modern}. We also note that a Turing instability itself is neither necessary for pattern formation \cite{brauns2020phase, al2022stationary} nor sufficient \cite{krause2024turing}, but it is the classical tool used to study patterning in these kinds of systems.}

\section{Characterizing transient instability}\label{Sec_characterizing}
{We now consider transient instability which can occur when the homogeneous state is asymptotically stable. This is typically studied through properties of non-normality and reactivity of matrices associated with the linearization of the system. We recall that a matrix $\mathbf{A}$ is non-normal if there does not exist a set of its eigenvectors which form an orthonormal basis of the underlying vectorspace. Reactivity of a matrix has several equivalent definitions, but essentially means that linear perturbations of a stable matrix may grow as algebraic transients before being dominated by exponential decay due to the negative real part eigenvalues.} 
} 

We consider a homogeneous equilibrium \((u_0, v_0)\) that is inherently stable for \(\jac\) but can undergo temporary destabilization due to the transient amplification of perturbations. This phenomenon arises from the reactive properties of the non-normal matrix \(\jac\). Specifically, if \(\jac\), which characterizes the linearized dynamics of the system, is non-normal for some \(k\), its eigenvectors do not form an orthogonal basis. Consequently, certain directions may be poorly represented, leading to significant amplification of components along those directions in the eigenvector basis. {The intuition for this is that lack of orthogonality in the eigenfunctions means that linear combinations of them can be small in norm, even if their coefficients are large, due to cancellations. This gives rise to polynomial growth in the transient stage, which is eventually dominated by exponential decay. } 

If, at \((u_0, v_0)\), \(\jac\) is also reactive (as not all non-normal matrices are reactive), this implies that even when \((u_0, v_0)\) is stable, perturbations around it can initially grow in norm before eventually decaying. This behavior facilitates transient deviations from equilibrium. Such transient growth is capable of driving the system out of equilibrium by first leaving the linear regime (where, for sufficiently large times, transient effects vanish as the eigenvalues dominate the system's evolution) and then by allowing nonlinearities in the kinetics to translate these deviations into stable non-homogeneous patterns {(assuming such states exist)}. The transient growth of perturbations thus provides a mechanism for the system to transition toward heterogeneous attractors (despite being asymptotically stable), which we refer to as patterns of reactivity.

\subsection{Measure of non-normality}
The  eigenvalues {of $\jac$} are given by 
\[
\lambda_{\pm} (k^2)= \frac{1}{2} \left( f_u + g_v - k^2 (D_u + D_v) \pm \sqrt{(f_u + g_v - k^2 (D_u + D_v))^2 - 4 h(k^2)} \right),
\]
with corresponding eigenvectors:
$
\mathbf{v}_{\pm}({\boldsymbol k}) = \frac{\mathbf{\tilde{v}}_{\pm}}{\|\mathbf{\tilde{v}}_{\pm}\|},
$
where:
\[
\mathbf{\tilde{v}}_{\pm}(k^2) =
\begin{bmatrix}
      f_u - g_v - k^2 D_u + k^2 D_v \pm \sqrt{(f_u + g_v - k^2 (D_u + D_v))^2 - 4 h(k^2)} \\\\
      2 g_u
\end{bmatrix}.
\]
 The degree of non-normality of the matrix of eigenvectors \( V_{k} = [\mathbf{v}_{+}(k^2), \mathbf{v}_{-}(k^2)] \) can be measured by evaluating how parallel the eigenvectors \( \mathbf{v}_{\pm}(k^2) \) are or, equivalently, how close the orthogonal complement of \( \mathbf{v}_{+}(k^2) \) is to being orthogonal to \( \mathbf{v}_{-}(k^2) \). We then introduce the quantity 
\[
\delta(k^2) = |\langle \mathbf{v}_+^\perp, \mathbf{v}_- \rangle| = \frac{4 \, |g_u|\, \sqrt{|(f_u-g_v+k^2(D_v-D_u))^2+4\,g_u\,(f_v+k^2\,\beta\,\ell(u_0))|}}{\|\mathbf{\tilde{v}}_+\| \,\|\mathbf{\tilde{v}}_-\|},
\]
which is equal to $1$ when the vectors \( \mathbf{v}_{+}(k^2) \) and \( \mathbf{v}_{-}(k^2) \) are orthogonal and vanishes when 
 they are perfectly parallel (in that case $\jac$ is not diagonalizable,  a case we do not consider in our analysis).

{We compute the denominator of this quantity as,} 
$$
\|\mathbf{\tilde{v}}_+\| \,\|\mathbf{\tilde{v}}_-\|= 4 |g_u|\, \sqrt{(f_u-g_v+k^2(D_v-D_u))^2+ (g_u+ f_v\,+ k^2\,\beta \, \ell(u_0))^2}
$$ so that 
\begin{equation}\label{eq:non_normality}
    \delta(k^2) = \sqrt{\frac{|(f_u-g_v+k^2(D_v-D_u))^2+4\,g_u\,(f_v+k^2\,\beta\,\ell(u_0))|}{(f_u-g_v+k^2(D_v-D_u))^2+ (g_u+ f_v +\,k^2\,\beta \, \ell(u_0))^2}}.
\end{equation}

The maximum non-normality occurs for values of $k^2$ corresponding to the minimum of $\delta(k^2)$, which vanishes when $\jac$ is not diagonalizable and assumes a value of $1$ when $\jac$ is normal.

\subsection{Reactivity as initial amplification}
\label{sec:react_in}
 The dynamics of perturbations of magnitude \( \|\widetilde{\mathbf{w}}_0\| \) of an equilibrium {are} determined by the solution of the linearized system (\ref{eq:general_linear}). Specifically, the amplification envelope is defined as:
\[
\rho_{k}(t) = \displaystyle \max_{\|\widetilde{\mathbf{w}}_0\|\neq 0} \frac{\|e^{\jac t} \widetilde{\mathbf{w}}_0 \|}{\|\widetilde{\mathbf{w}}_0\|} = \|e^{\jac t}\|=\|V_k\, e^{\Lambda_k\, t} \, V_{k}^{-1}\|,
\]
where $\Lambda_{k}$ is the diagonal matrix of the eigenvalues {and $V_{k}$ the matrix of eigenvectors}. 
It is well known (see for example \cite{trefethen2005spectra}) that the spectral abscissa \( \alpha(\jac) \) characterizes the behavior of the derivative of \( \rho_{k}(t) \) at the asymptotic limit \( t \to \infty \):
\[
\lim_{t \to \infty} \frac{d \rho_{k}(t)}{dt} = \lim_{t \to \infty } \frac{d}{dt} \| e^{t\, \jac} \| = \lim_{t \to \infty} t^{-1} \log \| e^{t\, \jac} \| = \alpha(\jac)=\mathrm{Re}(\lambda_+({k^2})).
\]
{ Conversely, the initial behavior, determined by the limit of the derivative of \( \rho_{k}(t) \) as \( t \to 0 \) is characterized\footnote{in the Hilbert space case \cite[Eq.~(14.2), p.~137]{trefethen2005spectra}.} by the largest eigenvalue of the Hermitian part of $\jac$, i.e.  \
\[
\lim_{t \to 0} \frac{d \rho_{k}(t)}{dt} = \lim_{t \to 0 } \frac{d}{dt} \| e^{t\, \jac} \| = \lim_{t \to 0} t^{-1} \log \| e^{t \, \jac} \| = r(\jac)
 = \max \lambda(H(\jac)),
 \]
 where $H(\jac) {= (\jac + \jac^H)/2}$. }
In general the following relation occurs
$$
e^{\alpha(\jac)\,t} \leq \|e^{t\, \jac} \| \leq e^{r(\jac)\,t}.
$$
 When \( r(\jac) > 0 \), perturbations always exhibit initial growth, and, due to the non-normal nature of \(\jac\), they can be amplified sufficiently to overcome the barrier separating the attraction basins of the homogeneous stable equilibrium and non-homogeneous states.

 {The quantity $r(\jac)$, also known as the \textit{numerical abscissa} \cite{trefethen2005spectra} or \textit{logarithmic norm} \cite{moler2003nineteen} in different research fields, is here referred to as \textit{reactivity}, following its introduction in ecological systems \cite{neubert1997alternatives} as an alternative measure to the sole focus on resilience}
 (associated with asymptotic behavior), which was later studied as a necessary condition for the local dynamics of reaction-diffusion systems that give rise to Turing patterns \cite{NEUBERT2002}. {We remark that \cite{neubert1997alternatives} originally defined reactivity in the context of the spectrum of the Hermitian part of the linear system; see Equations (10)-(13) there for an elementary justification of this definition.} This concept was further generalized in the work \cite{mari2017generalized} and more recently applied in this generalized context to {(ODE)} carbon dynamics models in the study \cite{diele2023soc}.

Before starting the analysis, we provide some results 
 which will be useful in what follows. Let us consider $J_{0}$, which corresponds to $\jac$ with ${k} ={0}$ ({see Equation \ref{eq:j_matrix}}).
\begin{proposition}\label{prop:J*reactive} If \(J_{0} \) has a negative trace, then \( J_{0} \) is reactive iff the following condition holds:
\[
f_u g_v - \frac{(f_v + g_u)^2}{4} < 0.
\]
 \end{proposition}
 \begin{proof}
 For \( J_{0} \) to be reactive {(i.e.~$r(J_{0})>0$)} , it is sufficient to require that the Hermitian part of \( J_{0} \),
\[
H(J_{0}) = \left[ \begin{array}{cc}
         f_u & \frac{f_v + g_u}{2} \\\\
         \frac{f_v + g_u}{2} & g_v
    \end{array} \right],
\]
has at least one positive eigenvalue. By applying Descartes' rule of signs to the characteristic polynomial
\[
\lambda^2 - \lambda \, (f_u + g_v) + f_u g_v - \frac{(f_v + g_u)^2}{4},
\]
and taking into account that the trace \( f_u + g_v < 0 \), the result follows.
 \end{proof}
 Notice that, when  \( J_{0} \) has negative trace and it is  not reactive, then,
\[
f_u g_v > \frac{(f_v + g_u)^2}{4} > 0.
\]
Then, \( f_u \) and \( g_v \) must have the same sign, and since \( J_{0} \) has a negative trace, both are negative. Consequently,
\[
A = D_u g_v + D_v f_u < 0.
\]

 To determine the general conditions under which reactivity is positive, we evaluate the Hermitian part of $\jac${, which is given by}
$$
    H\left(\jac\right) = \, \left [ \begin{array}{cc}
         f_u-k^2\, D_u& 
         \frac{(f_v+g_u)+k^2\,\beta \ell(u_0) }{2} \\\\
         \frac{(f_v+g_u)+k^2\, \beta \ell(u_0) }{2} & g_v-k^2 D_v
    \end{array}
    \right],
$$
and find conditions for the largest eigenvalue to be positive for some $k$.
 The characteristic polynomial is given by 
 \begin{equation}
 \label{eq:charpol2}
    \lambda^2 + \lambda \left( \left(D_u+D_v \right)k^2 - \left(f_u+g_v\right) \right) + \tilde h(k^2),
 \end{equation}
with $\tilde h(k^2)\,=\,h(k^2)- \left(\frac{f_v-g_u\,+\, \beta \ell(u_0)\, k^2}{2}\right)^2$.

{To ensure positive reactivity, and hence transient instability,} we want that the real part of at least one root of the  polynomial \eqref{eq:charpol2} is positive for  some $k \neq 0$. According to Descartes' rule of signs, we need \(\tilde h(k^2) < 0\) for some \(k \neq 0\) to assure the existence of a positive (real) solution of (\ref{eq:charpol}).

 From the definition of \(\tilde{h}(k^2)\), it is clear that if \(\jac\) is unstable, i.e., \(h(k^2) < 0\) for some \(k\neq 0\), then \(\tilde{h}(k^2) < 0\), making \(\jac\) reactive. Considering that the sign of the constant term of \(\tilde{h}(k^2)\) determines whether \(\jac\), when assumed asymptotically stable, is reactive or not, we explicitly write the expression of \(\tilde{h}(k^2)\) as:
 $$
    \tilde h(k^2) =  \left(D_u D_v -\frac{\beta^2 \ell^2(u_0)}{4}\right)k^4  - \left(A + \frac{(f_v+g_u)\beta \ell(u_0) }{2}\right) k^2 + f_u g_v -\frac{(f_v+g_u)^2}{4},
    $$
    and we distinguish the following cases:

\begin{enumerate}
\item $0\leq D_u D_v< \frac{\beta^2 \ell^2(u_0)}{4}$. In this case, \( \tilde h(k^2) \) is a second-order polynomial with a negative coefficient for the quadratic term. If the discriminant is negative, then \( \tilde{h}(k^2) < 0 \) for all \( k^2 \). Otherwise, if the discriminant is positive, then \( \tilde{h}(k^2) < 0 \) for some \( k^2 > 0 \), and no further conditions are necessary.

\item $D_u D_v \geq \frac{\beta^2 \ell^2(u_0)}{4}\geq 0$ and $J_{0}$ reactive.
   In this case, \( \tilde h(k^2) \) is a second-order polynomial in \( k^2 \) with a positive coefficient for the quadratic term and a discriminant given by

\[
\tilde{\Delta} = \left(A + \frac{(f_v + g_u)\beta \ell(u_0)}{2}\right)^2 - \left(D_u D_v - \frac{\beta^2 \ell^2(u_0)}{4}\right)(4 f_u g_v - (f_v + g_u)^2),
\]
which is positive. According to Descartes' rule of signs, the polynomial \( \tilde{h}(k^2) \) admits at least one positive root, thereby ensuring the existence of a range of wavenumbers \( k^2 \) for which \( \tilde{h}(k^2) \) takes negative values.

 \item $D_u D_v \geq \frac{\beta^2 \ell^2(u_0)}{4}>0$ and $J_{0}$ not reactive. 
In this case, we must require that \( \tilde{\Delta} > 0 \); otherwise, the polynomial would be positive for all values of \( k^2 \). Moreover, since the existence of a positive root corresponds to a variation in the signs of the coefficients of \( \tilde{h}(k^2) \), we also have to require that
\begin{equation}\label{eq:k2minimo2}
  - \frac{(f_v+g_u)\beta \ell(u_0) }{2} < A <0.
\end{equation}
where $A<0$ ($J_{0}$ is assumed stable and not reactive). 
\end{enumerate}

 The range of reactive wavenumbers $k^2$ are determined by the roots 
\begin{equation}
\label{eq_range_wave_react}
    \begin{aligned}
    \tilde k_m &:= \frac{A +\frac{(f_v+g_u)\beta \ell(u_0) }{2} - \sqrt{\tilde \Delta}}{2\left(D_u D_v -\frac{\beta^2 \ell^2(u_0)}{4}\right)}, \quad \tilde k_p :=\frac{A +\frac{(f_v+g_u)\beta \ell(u_0) }{2}+\sqrt{\tilde \Delta}}{2 \left(D_u D_v -\frac{\beta^2 \ell^2(u_0)}{4}\right)} 
    \end{aligned}.
\end{equation}
 Hence, in the first case, $\jac$ is reactive for all values of $k^2$ if $\tilde{\Delta} < 0$, and for values $0 < k^2 < \min \{ \max \{0,\tilde{k}_p\},\max \{0,\tilde{k}_m\}\}$ or $k^2 > \max \{ \max \{0,\tilde{k}_p\},\max \{0,\tilde{k}_m\}\}$. In the second case, $\jac$ is reactive for 
$0< k^2 < \tilde{k}_p$. In the third case, $\jac$ is reactive for 
$\tilde{k}_m < k^2 < \tilde{k}_p$. 

Finally, we summarize the conditions for diffusion-chemotaxis-driven initial instability as follows:

\begin{proposition} \label{prop_15}
    Under the hypothesis of $J_{0}$ stable then, $\jac$ is reactive  if any of the following conditions hold:
$$
\begin{cases}
  & |\beta\, \ell(u_0)|> 2\,A_1, \quad {\textrm{(Case 1)}}, \\\\
  &|\beta\, \ell(u_0)|\leq 2\,A_1, \quad \text{and} \quad  J_{0} \quad\text{is reactive } (A_3>2A_2), \quad {\textrm{(Case 2)}}\\\\
   &|\beta\, \ell(u_0)|\leq 2\,A_1 \quad \text{and} \quad  J_{0} \quad \text{is not reactive }  (A_3\leq 2A_2),\quad \text{and}\\\\
   & {\,\sqrt{A_1^2 -\frac{\beta^2\, \ell(u_0)^2}{4}}} \, \sqrt{4A_2^2 - A_3^2} -\frac{(f_v + g_u)\beta\, \ell(u_0)}{2}  < A < 0, \quad {\textrm{(Case 3)}},
\end{cases}
$$
where  \( A= D_u g_v + D_v f_u  \), \( A_1 = \sqrt{D_u D_v}\geq 0 \), \( A_2 = \sqrt{\det(J_{0})} >0\) and \(A_3 = |f_v - g_u| \geq 0\).
\end{proposition}
\begin{proof}
 We only need to analyze the case when \(|\beta\, \ell(u_0)|\leq 2\,A_1\), and $J_{0}$ is  not reactive {corresponding to case 3 above, as the other cases immediately ensure reactivity (as shown above)}. Imposing \( \tilde{\Delta} > 0 \) requires
\[
\left( A + \frac{(f_v + g_u)\beta \ell(u_0)}{2} \right)^2 > \left( D_u D_v - \frac{\beta^2 \ell^2(u_0)}{4} \right) \left( 4 f_u g_v - (f_v + g_u)^2 \right).
\]
By imposing \eqref{eq:k2minimo2}, we can equivalently express the above condition as
\[
A + \frac{(f_v + g_u)\beta \ell(u_0)}{2} > \sqrt{D_u D_v - \frac{\beta^2 \ell^2(u_0)}{4}} \sqrt{4 f_u g_v - (f_v + g_u)^2}, \quad A < 0.
\]
Using the adopted notation and noting that
\[
4 f_u g_v - (f_v + g_u)^2 = 4A_2^2 - A_3^2>0 \quad \iff A_3< 2A_2,
\]
the result follows.
\end{proof}

The above proposition shows how chemotaxis, diffusion, and local dynamics act in determining the reactivity of the system. The term $|\beta \ell(u_0)|$ represents the strength of chemotaxis, while $A_1 = \sqrt{D_u D_v}$ reflects the stabilizing effects of diffusion. Condition 1 emphasizes that strong chemotaxis ($|\beta \ell(u_0)| > 2A_1$) can independently drive reactivity by overcoming the stabilizing influence of diffusion. Condition 2 demonstrates that moderate chemotaxis ($|\beta \ell(u_0)| \leq 2A_1$) can still result in reactivity if the local Jacobian $J_{0}$ is inherently reactive. Condition 3 represents a finely tuned scenario where both chemotaxis and the local dynamics are weak. In this case, reactivity arises from a specific balance between diffusion, chemotaxis, and local dynamics. {An interesting observation here, related to the implication of reactivity in Cases 1 and 3 above, is that the reactivity of the Jacobian in the absence of transport ($J_{0}$) is not necessary to drive reactivity of $\jac$, or even of patterning instabilities, as it is in the reaction-diffusion case \cite{NEUBERT2002}.}

    

 Detecting reactivity regions within the stability region of a diffusion-chemotaxis model will indicate the potential for the emergence of {patterns which arise directly due to non-normality.}

\subsection{The amplification envelope}
Reactivity is a measure of solution behavior as $t \to
0$, and thus complements stability, which describes
solution behavior as $t \to \infty$. For {a} non-normal Jacobian, neither describes all the transient behavior between zero and infinity. If $\jac$ is reactive and solutions
can grow in magnitude, we can ask how large a perturbation can possibly get, and how long growth can
continue.  For {a} {non-normal} Jacobian, this transient behavior is not described  as $t \to
0$, or $t \to \infty$  but by the amplification envelope curve $\rho_{k}(t)=\|e^{\jac \, t}\|$ at  intermediate values of $t$ \cite{neubert1997alternatives}. 

 In the book by \cite{trefethen2005spectra}, several bounds of $\rho_{k}(t)=\|e^{\jac \, t}\|$ can be found. For example, an estimate related to the non-normality of the matrix $\jac$ through the condition number $\mu(V_{k})$ of the matrix of eigenvectors and to the largest eigenvalue   $\lambda_+(k^2)$ is given by the upper bound
\[
\rho_{k}(t) =  \|e^{\jac \, t} \| = \|V_{k}\, e^{\Lambda_{k}\, t} \, V_{k}^{-1}\| \leq \|V_{k}\|\, \| V_{k}^{-1}\| \, \|e^{\Lambda_{k}\, t}\| \,=\, \mu(V_{k}) \, e^{\mathrm{Re}(\lambda_+({k}))\, t}.
\]
This bound provides a reference point, but for sharper information we turn to the Kreiss constant $\mathcal{K}(\jac)$,
\begin{equation}\label{kreiss}
    \rho_{k}(t) \geq \mathcal{K} (\jac)=\sup_{\epsilon >0}\frac{\alpha_\epsilon (\jac)}{\epsilon}\, ,
\quad \quad \forall t\geq 0,
\end{equation}
where $\alpha_\epsilon(\jac)$ denotes the $\epsilon$-pseudospectral abscissa of $\jac$ \cite{trefethen2005spectra}. Moreover, in the specific case of {the} $2$ dimensional matrix $\jac$ it holds that
$$
\mathcal{K} (\jac) \leq \rho_{k}(t) \leq 2\, e  \, \mathcal{K} (\jac), \quad \quad \forall t\geq 0.
$$

 Another useful estimate relating \(\rho_{k}(t)\) to the measure of non-normality is provided in \cite{klika2017significance}. Under the assumption that \(\jac\) is stable, we define \(\Sigma = \lambda_+ - \lambda_-\), and it follows that \(\rho_{k}(t) \approx \chi(t)\), where 
\begin{equation}\label{eq:chi_real}
   \chi(t) := \delta^{-1} \left(e^{\lambda_+ t} - e^{\lambda_- t}\right), 
\end{equation}
with a maximum given by
\begin{equation}\label{eq:max_chi_real}
\chi^* = \delta^{-1} \left(\frac{\lambda_-}{\lambda_+}\right)^{\lambda_+ / \Sigma} \left(1 - \frac{\lambda_+}{\lambda_-}\right),
\end{equation}
at 
$
t^* = \frac{1}{\Sigma} \ln \frac{\lambda_-}{\lambda_+},
$
in the case where \(\lambda_- < \lambda_+ <0\) are real. Otherwise, when the eigenvalues are complex, we have
\[
\chi(t) := e^{\Re\lambda t} \left( \delta^{-1} \left[ 1 - e^{-2 \delta \Im\lambda t} \right] + 1 \right),
\]
where \(\Im\lambda = \Im\lambda_+ = -\Im\lambda_- > 0\) and \(\Re\lambda = \Re\lambda_- < 0\). In this case, the maximum of the envelope of \(\chi(t)\),
\[
\overline{\chi}(t) = 2 \delta^{-1} e^{\Re\lambda t} \lvert \sin(\Im\lambda t) \rvert,
\]
is reached at
\[
t^* = \frac{1}{\Im\lambda} \arctan\left(-\frac{\Im\lambda}{\Re\lambda}\right).
\]

 In the following we will exploit these estimates in order to relate their values to the onset of {patterns of reactivity} in the dynamics of a particular diffusion-chemotaxis system.

\section{Patterns of reactivity in the MOMOS Model}\label{Sec_patterns}

In this section, we focus on the reaction-diffusion chemotaxis MOMOS model \cite{hammoudi2018mathematical}.  A simplified version with only two compartments is considered: the microbial biomass and the soil organic matter {are} represented by the state variables $u$ and $v$ respectively. The model is then given by
\begin{equation}\label{eq:model}
\begin{aligned}
&\partial_t u  = D_u \Delta u -\beta \, \nabla \cdot( \ell(u) \nabla v) - k_1 \, u - q \, u^2 + k_2 \,   v, \\
&\partial_t v = D_v \Delta v + k_1 \, u \, -k_2 \, v  + \, c.
\end{aligned}
\end{equation}
Here the parameter $\beta$ is the chemotactic sensitivity, whereas $D_u$ and $D_v$ are the diffusion parameters of the microbial mass and of the soil organic matter, respectively. 
The function $l(\cdot)$ is the density-dependent chemotactic response. The parameter $k_1$ represents the microbial mortality rate, $k_2$ the soil carbon degradation rate, $q$ the metabolic quotient {(namely, the carbon respiration rate per unit biomass of bacteria)}, and $c$ the soil carbon input. All the parameters in  model (\ref{eq:model})  are assumed to be positive constants. The domain $\Omega$ is again taken as a square, and periodic boundary conditions, except in Section \ref{Sec_continuation} where Neumann (no-flux) boundary conditions are more convenient. Initial conditions are taken as small random perturbations of the homogeneous equilibrium $(u_0,v_0)${, which are normally distributed across the spatial domain with standard deviation given by $\eta$} . In the following, we will focus on the {simplest chemotactic sensitivity of linear density dependence given by} $\ell(u)=u$. 

The  model (\ref{eq:model}) admits the two constant solutions as spatially homogeneous equilibria: $\left(u_i, \displaystyle \frac{u_i(q\,u_i+k_1)}{k_2}\right)$, for $i=0,1$, where $u_0=\sqrt{\frac{c}{q}}$ and  $u_1=-\sqrt{\frac{c}{q}}$.
Due to biological considerations,  we will focus only on the unique positive and hence feasible spatially homogeneous equilibrium:  

\begin{equation*}
  P_0=(u_0,v_0)=\,\left( \sqrt{\dfrac{c}{q}}, \dfrac{k_1}{k_2} \sqrt{\frac{c}{q}} +\dfrac{c}{k_2} \right).
 \end{equation*}

 The spatially homogeneous solution $P_0$ is always linearly stable in the absence of diffusion ($D_u=D_v=0$) and chemotaxis ($\beta=0$). In fact, the Jacobian matrix of the reaction terms, evaluated at $P_0$,
\begin{equation*}
 \label{J_Momos}
     J_{0}=\left[ 
     \begin{array}{ccc}
        f_u 
        && f_v \\
        g_u  && g_v
     \end{array}\right]=\left[ 
     \begin{array}{ccc}
        -\,k_1 -2\sqrt{c\,q} 
        && k_2 \\
        k_1  && -k_2
     \end{array}\right],
 \end{equation*}
 has  $\det (J_{0})=2\,k_2\,\sqrt{cq}>0$ and ${\mathrm{tr}}(J_{0})=-2\,\sqrt{cq}-k_1-k_2 <0$.
 Equation 
(\ref{eq:suff_conditions_cd}) indicates that when 
\begin{equation}\label{eq:chem_cond_momos}
 -D_u k_2 - D_v (k_1+2 \sqrt{cq}) + k_1 \beta \sqrt{\frac{c}{q}} > 2\sqrt{ 2 D_u D_v k_2 \sqrt{cq}},     
 \end{equation}
 i.e. when 
 \begin{equation}\label{bifurcation_curve}
    \beta> \beta_c:= \displaystyle \frac{\sqrt{q}}{k_1\sqrt{c}}\left(D_u\,k_2+D_v\,k_1+2\,D_v\sqrt{cq}+\sqrt{8\,D_u\,D_v\,k_2\sqrt{cq}} \right),
    \end{equation}
    then the spatially homogeneous equilibrium $P_0$ of model (\ref{eq:model}) undergoes {a} chemotaxis-driven instability. Note that the emergence of patterns can be due only to the chemotaxis term as, in the absence of chemotaxis, i.e. for $\beta=0$, Equation (\ref{eq:chem_cond_momos}) cannot be satisfied.

 In \cite{noi}, symplectic techniques have been applied to numerically approximate the spatial patterns arising as non-homogeneous solutions of the MOMOS model (\ref{eq:model}) due to the asymptotic instability of the matrix $\jac$ {(i.e.~those arising from standard chemotaxis-driven Turing instabilities)}. Recent studies have also explored model order reduction or data-driven approaches for efficiently capturing complex spatio-temporal dynamics and Turing-type patterns in reaction-diffusion systems \cite{BMS2021, AMS2023, AMS2024}.

In this paper we are interested in detecting  patterns due to transient instability generated by the non-normal nature of 
\[
\jac = \begin{bmatrix}
         -k_1-2\sqrt{c q} - k^2 D_u & k_2 + k^2 \beta \sqrt{\frac{c}{q}}\, \\\\
         k_1 & -k_2 - k^2 D_v
    \end{bmatrix},
\]
when \( J_{0} \) is stable {and} \( \jac \) is reactive. From Proposition \ref{prop_15}, when \( J_{0} \) is stable \( \jac \) is reactive if any of the following conditions hold:
\begin{enumerate}
    \item \( \frac{|\beta|}{\sqrt{q}}> 2 \sqrt{\frac{D_u\, D_v}{c}} \),
    \item \( \frac{|\beta|}{\sqrt{q}}< 2 \sqrt{\frac{D_u\, D_v}{c}} \), \,\,and \( \, q^{1/4} > \frac{|k_2 - k_1|}{\sqrt{8 \, k_2 \, \sqrt{c }}} \),
    \item 
    \(
        \frac{|\beta|}{\sqrt{q}}< 2 \sqrt{\frac{D_u\, D_v}{c}}, \,\, q^{1/4} < \frac{|k_2 - k_1|}{\sqrt{8 \, k_2 \, \sqrt{c }}},
    \)\,\, and
    
    \(         \sqrt{D_u D_v - \beta^2 \frac{c}{4q}} \, \sqrt{8 \, k_2 \sqrt{c\, q}} - \frac{(k_2 + k_1)\, \beta \sqrt{\frac{c}{q}}}{2} < -k_2 D_u - D_v (k_1 + 2 \sqrt{c q}) < 0 \),
\end{enumerate}

The{se} three {cases} highlight the roles of the chemotaxis coefficient ($\beta$), the nonlinearity (quadratic)  parameter ($q$), and the diffusion coefficients ($D_u$ and $D_v$) in determining reactivity. We see that a  
strong chemota{ctic response, given by large $\beta$,} can independently induce reactivity (Case 1), while moderate or weak chemotaxis requires support from local gradients or a finely balanced system.
 Higher $q$ generally stabilizes the system by raising the thresholds for reactivity, resisting transient amplification.
 Diffusion stabilizes the system in all cases, opposing the destabilizing effects of chemotaxis and requiring larger $\beta$ or more significant local gradients for reactivity to arise.

 We fix parameters $D=D_u=D_v=0.6$, $k_1=0.4$ and $k_2=0.6$ and $c=0.8$ as in \cite{noi}. The bifurcation diagram in Figure \ref{fig:momos}(a)  shows the region in the \((\beta, q)\) parameter space where $\jac$ is both stable (for all $k$)  and reactive (for at least one $k$), and {hence} spatial pattern initiation {due to transient instability} may arise. {We remark that, besides a very small region for small $\beta$ and $q$ in the bottom-left (not shown given its size), which is due to Case 3, the remainder of the transient instability region is due to Case 1 above. }

 {We systematically ran simulations across one of the boundary data points in Figure \ref{fig:momos}(a) corresponding to $q = 0.0196639$. We ran simulations for varying values of $\beta$ and sizes of the initial random perturbation, given by $\eta$, for $T=10^5$ time units. We used the following functional (see \cite{berding1987heterogeneity}) as a proxy for spatial heterogeneity to differentiate between patterned and homogeneous states:    \begin{equation}\label{heterogeneity_functional}
        E(u) = \left(\int_\Omega ||\nabla u||^2d\mathbf{x}\right)^{\frac{1}{2}} = ||\nabla u||_{L^2},
    \end{equation}
    where the integral is over the spatial domain $\Omega = [0,L]$ in 1D, or $\Omega = [0,L] \times [0,L]$ in 2D. The results of these simulations are shown in Figure \ref{fig:momos}(b). We see that even small initial perturbations can lead to stable far-from-equilibrium patterns, despite the homogeneous state being linearly stable, though there is a nontrivial dependence on the size of the initial perturbation.}

\begin{figure}
    \centering
    \includegraphics[width=0.48\linewidth]{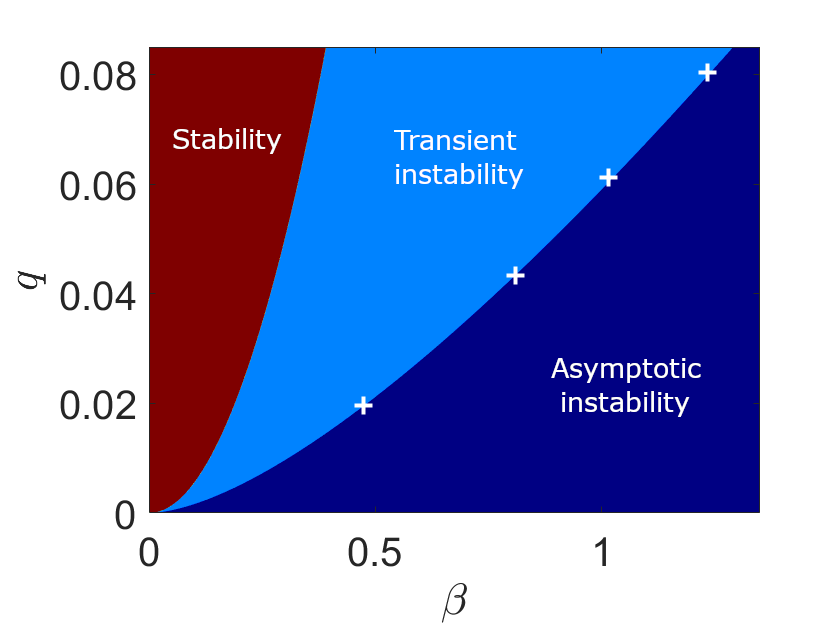}
    \includegraphics[width=0.48\linewidth]{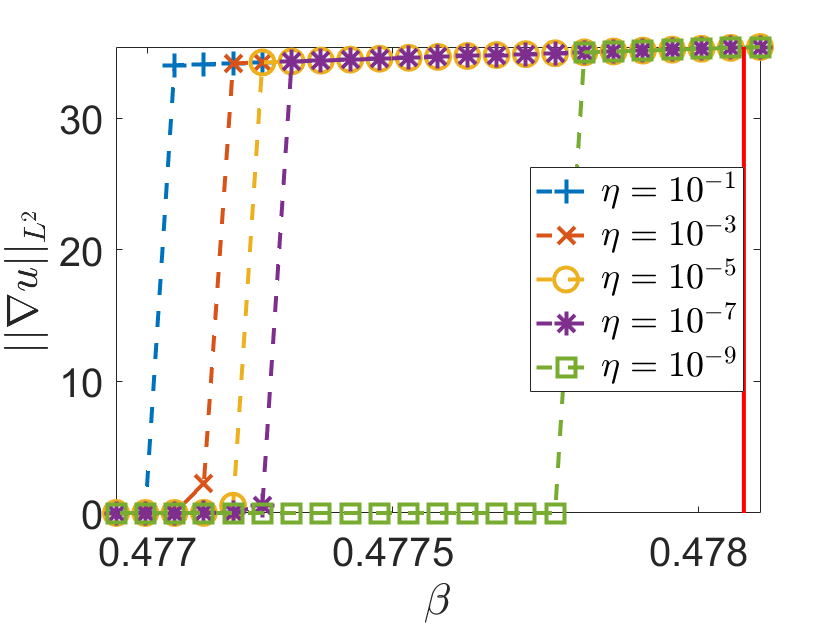}

            \hspace{0.4cm} (a) \hspace{5.7cm} (b) \hfill
            
    \caption{{(a)} Bifurcation diagram for the MOMOS model, \eqref{eq:model}, with parameters set to \( D_u = D_v = 0.6 \), \( k_1 = 0.4 \), \( k_2 = 0.6 \), and \( c = 0.8 \). Data points {(white crosses)} indicate parameter pairs where patterns of reactivity,  {arising due to non-normality}, are observed {at the parameters $(q, \beta) = (0.0196639, 0.474095), (0.0804361, 1.23535), (0.061122, 1.01668)$, and $(0.0433, 0.806)$. (b) A heterogeneity norm from systematic simulations across one of these data points for varying size of the initial data, $\eta$. The red vertical line indicates the Turing bifurcation point, where everything to the left of this line is in the Turing stable parameter region.}
}
    \label{fig:momos}
\end{figure}
\begin{figure}
     \centering    

\includegraphics[width=0.45\linewidth]{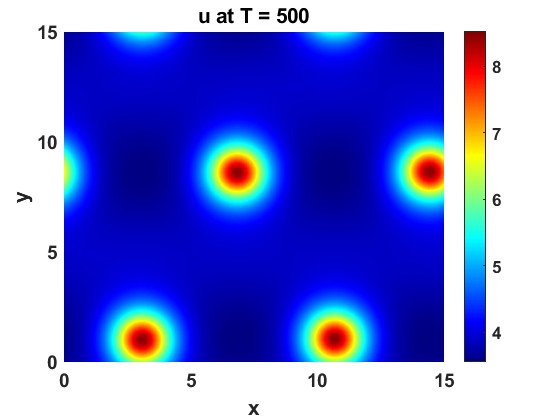}
\includegraphics[width=0.45\linewidth]{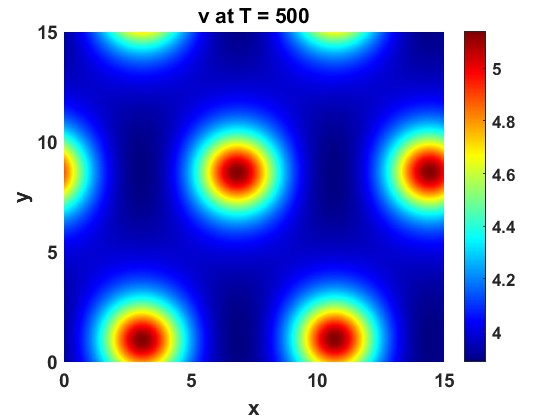}
   \caption{Patterns of reactivity {arising due to non-normality} in the MOMOS model, \eqref{eq:model}, for $q=0.0433$, $\beta=0.806$. The other parameters are set to \( D_u = D_v = 0.6 \), \( k_1 = 0.4 \), \( k_2 = 0.6 \), and \( c = 0.8 \).
}
    \label{fig:momos2}
\end{figure}

In Figure \ref{fig:momos2} {we give an example long-time simulation of a pattern} found for $q=0.0433$ and $\beta=0.806$, which lies in the region where $\jac$ is stable ($\beta < \beta_c \approx 0.8087$) and reactive according to the above Case 1  ($\beta > 2 D \sqrt{\frac{q}{c}} \approx 0.279$).  We used the implicit-symplectic IMSP\_IE scheme introduced in \cite{settanni2016devising, diele2019geometric} and exploited in \cite{noi} on a square domain of length $L=15$ discretized with spatial stepsize $h_x=0.2$ on a temporal interval $[0,500]$ discretized with $h_t=0.01$. 
The Matlab code used to generate patterns in Figures \ref{fig:momos2}, \ref{fig:momos3} can be found at \href{https://github.com/CnrIacBaGit/Patterns-of-Reactivity-MOMOS-}{this GitHub repository}\footnote{\url{https://github.com/CnrIacBaGit/Patterns-of-Reactivity-MOMOS-}}.

{We found a range of similar patterns existing all along the boundary between transient and asymptotic instability, i.e.~the boundary of the Turing space, using the above code. We also separately implemented a central finite-difference code using Matlab's \textsc{ode15s} function, which implements BDF15 for adaptive timestepping, in order to ensure that these solutions are found robustly in this region of the parameter space. We also implemented the model using the web-based solver, VisualPDE \cite{walker2023visualpde}, and this can be found at \href{https://github.com/AndrewLKrause/Transient-Instability-Patterning-Codes}{this GitHub repository}\footnote{\url{https://github.com/AndrewLKrause/Transient-Instability-Patterning-Codes}}, along with the finite-difference implementation of the model.}

{We remark that stable patterned states in the transiently unstable regime were only observed in two-dimensional models on sufficiently large domains; simulations in one spatial dimension only exhibited patterns in the Turing unstable regime. We explore this via numerical continuation after discussing the onset of instabilities in the transiently unstable regime.}

\subsection{Analysis of transient instability} In order to analyze the onset of patterns of reactivity due to non-normality for the {parameters} $q=0.0433$ and $\beta=0.806$, we evaluate the measure of non-normality, $\delta(k^2)$ given in (\ref{eq:non_normality}).

In Figure~\ref{fig:non-normality} on the left, we plot the polynomials $h$ and $\tilde h$ as function of $k^2$. As we are considering Case 1, $\tilde h(k^2)$ has a negative coefficient for the quadratic term, and the discriminant of $\tilde h(k^2)$ is $4.5302 > 0$. Hence, the wavenumbers that guarantee the reactivity of $\jac$ are $k^2 > \tilde k_m = 0.1173$. The measure of non-normality for values starting from $\tilde k_m$ is shown in Figure~\ref{fig:non-normality}, where we observe that $\delta(k^2)$ decreases as $k^2$ increases, thus confirming that $\jac$ is not normal and indicating an increase in non-normality with increasing $k^2$. 

 To evaluate the amplification envelope $\rho_{k}(t)$, we first observe that the discriminant in (\ref{eq:hmin}) satisfies $\Delta \approx -0.0053 < 0$. This implies that $h(k^2) > 0$ for all $k^2$, which, in turn, indicates that the eigenvalues $\lambda_{\pm}$ of $\jac$ are real and negative. In Figure \ref{fig:ampliI},
the time evolution of the amplification envelope $\rho_{k}(t)$ is shown alongside its theoretical estimate $\chi(t)$ in (\ref{eq:chi_real}) and its maximum value $\chi^*$ in (\ref{eq:max_chi_real}) for wavenumbers in the range $0.2 \leq k^2 \leq 1$. The plots demonstrate the close agreement between $\rho_{k}(t)$ and $\chi(t)$, validating $\chi(t)$ as a predictor of transient dynamics with $\chi^*$ as lower bound for the observed maximum of $\rho_{k}(t)$. The transient amplification becomes more sustained as $k^2$ increases, peaking at $k^2=1$ with an estimated $\chi^* = 1.7108$. The minimum of $h(k^2)$ ($k^2$ near $0.7812$) exhibits the largest negative eigenvalue of $\jac$ and the maximal return time (i.e. the duration for $\rho_{k}(t)$ to decay back to its initial value); see Figure \ref{fig:non-normality}. As $k^2$ increases further, $\rho_{k}(t)$ begins to decrease more rapidly, indicating faster stabilization of the system.

 In Figure \ref{fig:ampliII} the analysis is extended to a wider range of wavenumbers, specifically $10 \leq k^2 \leq 10^5$. For larger values of $k^2$, the return time continues to decrease significantly, reflecting the rapid stabilization of the system at high wavenumbers. The amplification dynamics become more pronounced, as evidenced by the higher peak values of $\rho_{k}(t)$. Despite the increased amplification, its maximum stabilizes at an estimated value of $\chi^* = 2.1242$ for large $k^2$.

 Figures \ref{fig:ampliI} and \ref{fig:ampliII} highlight distinct behaviors in the transient amplification and return times across different ranges of $k^2$. In Table \ref{tab:amplification_return_time}, we present the return time, which exhibits a significant increase as $k^2$ approaches the critical value of $0.7812$, corresponding to the largest negative eigenvalue of $\jac$ nearing zero. This behavior highlights a slower stabilization near this critical value. In contrast, for increasing values of $k^2$, the return time decreases sharply, indicating faster stabilization at larger wavenumbers{, as one would expect due to the asymptotics of the dispersion relation suppressing high-wavenumber dynamics}.
\begin{figure}
    \centering
    \includegraphics[width=0.45\linewidth]{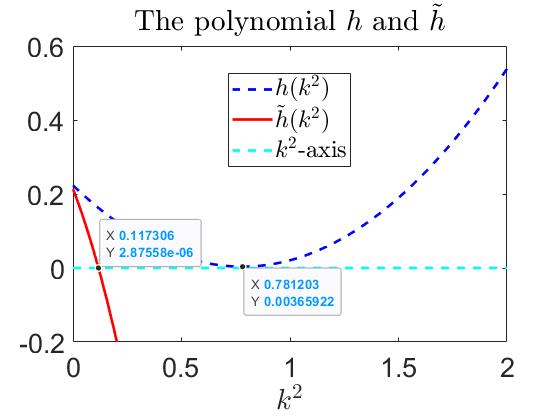}
      \includegraphics[width=0.45\linewidth]{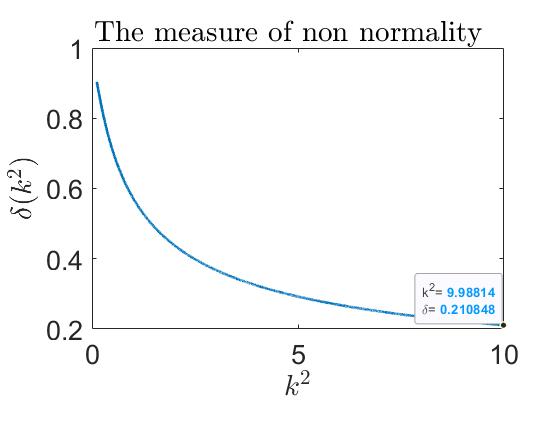}
    \caption{On the left, the polynomials $h(k^2)$ and $\tilde{h}(k^2)$ are plotted as functions of $k^2$. For Case 1, $\tilde{h}(k^2)$ has a negative quadratic coefficient with positive discriminant, resulting in reactivity of $\jac$ for $k^2 > \tilde{k}_m = 0.1173$. At $k^2 = 0.7812$, corresponding to the minimum of $h(k^2)$, the largest negative eigenvalue of $\jac$ approaches zero. On the right, the measure of non-normality $\delta(k^2)$ is shown for values of $k^2 \geq \tilde{k}_m$, demonstrating that $\delta(k^2)$ decreases as $k^2$ increases. This confirms that $\jac$ is non-normal, with non-normality becoming more pronounced at higher $k^2$.}
    \label{fig:non-normality}
\end{figure}

\begin{figure}
    \centering
    \includegraphics[width=0.45\linewidth]{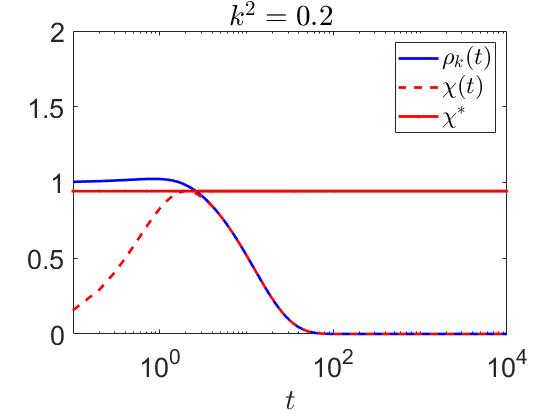}
      \includegraphics[width=0.45\linewidth]{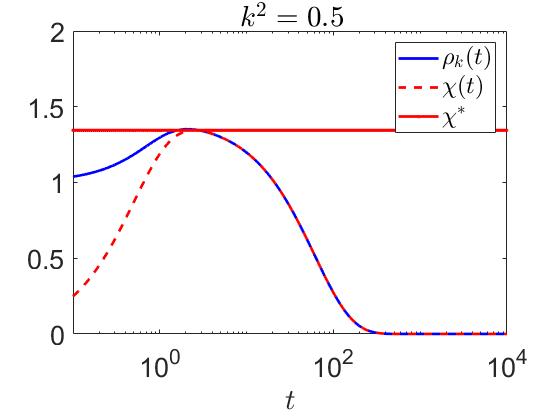}
        \includegraphics[width=0.45\linewidth]{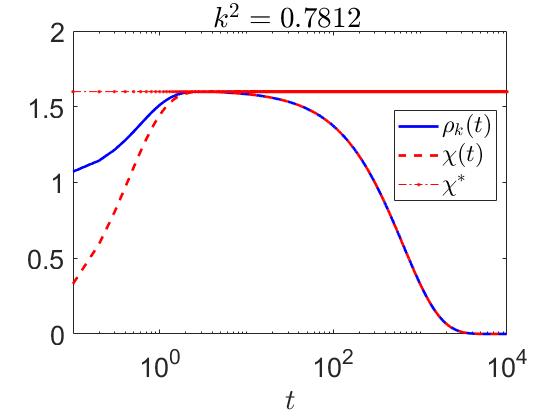}
          \includegraphics[width=0.45\linewidth]{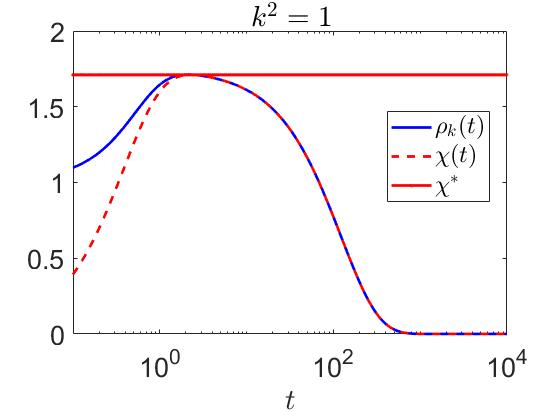}
    \caption{Time evolution of the amplification envelope $\rho_{k}(t)$, its theoretical estimate $\chi(t)$, and its maximum value $\chi^*$ for wavenumbers $0.2 \leq k^2 \leq 1$. The plots demonstrate a close agreement between $\rho_{k}(t)$ and $\chi(t)$, confirming the predictive capability of $\chi(t)$ for transient dynamics. The amplification becomes more sustained as $k^2$ increases, reaching an estimated maximum of $\chi^* = 1.7108$ at $k^2=1$. At $k^2 = 0.7812$, corresponding to the minimum of $h(k^2)$, the return time—defined as the time for $\rho_{k}(t)$ to decay back to its initial value—is maximized. Beyond this point, larger $k^2$ values result in a more rapid decline of $\rho_{k}(t)$, indicating faster stabilization of the system.
}
    \label{fig:ampliI}
\end{figure}
\begin{figure}
    \centering
    \includegraphics[width=0.45\linewidth]{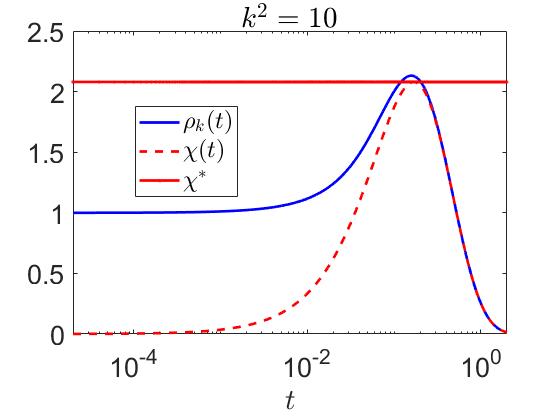}
      \includegraphics[width=0.45\linewidth]{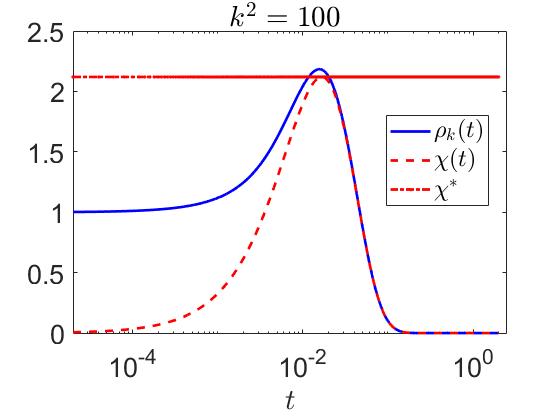}
        \includegraphics[width=0.45\linewidth]{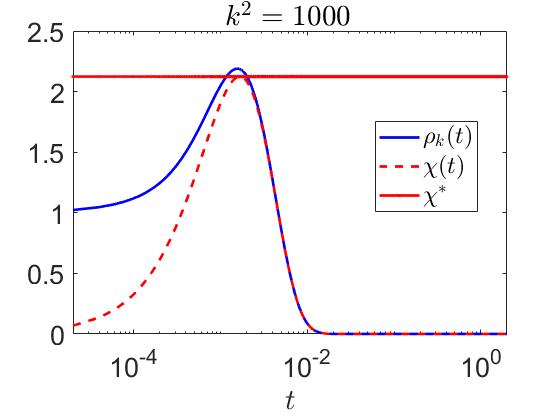}
          \includegraphics[width=0.45\linewidth]{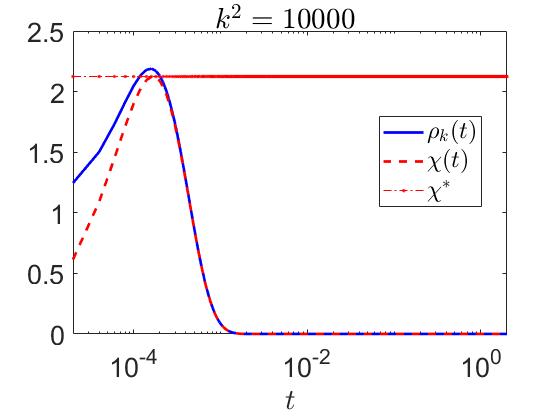}
    \caption{Time evolution of $\rho_{k}(t)$, $\chi(t)$, and $\chi^*$ for higher wavenumbers $10 \leq k^2 \leq 10^5$. As $k^2$ increases, the return time decreases significantly, showing the rapid stabilization of the system at larger wavenumbers. The transient amplification becomes more pronounced, with $\rho_{k}(t)$ achieving higher peak values. However, the maximum amplification stabilizes at $\chi^* = 2.1242$ for sufficiently large $k^2$, suggesting the existence of a maximum possible amplification value.}
    \label{fig:ampliII}
\end{figure}

\begin{table}[ht]
\centering
\caption{Values of \( \max(\rho_{k}(t)) \), the estimate \( \chi^* \) in (17), and the return time for the different values of \( k^2 \).}
\label{tab:amplification_return_time}
\vspace{0.2cm}
\begin{tabular}{|c|c|c|c|}
\hline
\( k^2 \) & \( \max(\rho_{k}(t)) \) & \( \chi^* \) & Return time \\ \hline
0.2      & 1.0239                     & 0.9423       & \( 0.85 \)   \\ \hline
0.5      & 1.3523                     & 1.3442       & \( 18.90 \)   \\ \hline
0.7812   & 1.6001                     & 1.5996       & \( 296.933 \)   \\ \hline
1        & 1.7130                     & 1.7108       & \( 66.44 \)   \\ \hline
10       & 2.1319                     & 2.0794       & \( 0.3669 \)   \\ \hline
$10^2$   & 2.1839                     & 2.1196       & \( 0.0312 \)   \\ \hline
$10^3$   & 2.1893                     & 2.1237       & \( 0.0311 \)   \\ \hline
$10^4$   & 2.1898                     & 2.1241       & \( 3.2\,10^{-4} \)   \\ \hline
$10^5$   & 2.1899                     & 2.1242       & \( 4\,10^{-5} \)   \\ \hline
\end{tabular}
\end{table}

 Conversely, the maximum amplification $\max_t(\rho_{k}(t))$ grows with increasing $k^2$, stabilizing at approximately $2.1899$. The estimate $\chi^*$ provides a reliable lower bound for the maximum amplification, stabilizing at $2.1242$, with a difference of $0.0657$ compared to the maximum amplification. In this context, we fix $k^2 = 0.7812$ and compare the lower bound $\chi^*$ with the Kreiss constant $\mathcal{K}(\jac)$, estimated by exploiting the definition in (\ref{kreiss}). We used \textsc{EigTool} \cite{wright2002eigtool} to evaluate the $\epsilon$-pseudospectral abscissa for several values of $\epsilon$. In Table \ref{tab:updated_values} we estimate the ratio $\alpha_\epsilon / \epsilon$ and found that its maximum is attained at approximately $\mathcal{K}(\jac) = 1.615668$. Comparing this with the value provided by $\chi^*$, we observe that $\chi^*$ gives a lower bound that is much closer to the maximum amplification.

\begin{table}[ht]
\centering
\begin{tabular}{|c|c|c|}
 \hline
$\epsilon$ & $\alpha_\epsilon$ & $\alpha_\epsilon/\epsilon$ \\  \hline
0.01 & 0.01443 & 1.4433 \\  \hline
0.04 & 0.06171 & 1.5429 \\  \hline
0.05 & 0.07723 & 1.54477 \\  \hline
0.051 & 0.07878& {1.54479} \\  \hline
0.052 & 0.08032  & 1.54478 \\  \hline
0.06 & 0.09264 & 1.5441 \\  \hline
0.1  & 0.15322  & 1.5322 \\
 \hline
\end{tabular}
\caption{Values of $\epsilon$, pseudo-abscissa $\alpha_\epsilon$, and their ratio $\alpha_\epsilon / \epsilon$ for $k^2 = 0.7812$. The maximum ratio,  highlighted in bold, corresponds to the estimate of the Kreiss constant.}
\label{tab:updated_values}
\end{table}

\subsection{Detecting stable reactivity patterns}

Throughout the region of reactivity, we can expect the emergence of transient patterns. However, from the previous analysis, we have understood that two key ingredients will be necessary for their emergence: sufficient amplification and a sufficiently long return time. {Importantly, this transient analysis is necessary but not sufficient for patterns to remain indefinitely, as this will require multistability of homogeneous and patterned states, which we explore in the next section.}

 Building on the {first} two pillars, here we aim to identify a threshold for the maximum amplification and return time that enables the emergence of stable reactive patterns. To achieve this, for each pair \((q, \beta)\) in the bifurcation diagram, we selected \(k^2\) and estimated the maximum amplification using the lower bound \(\chi^*\). Specifically, \(k^2\) is chosen within the range of reactivity at the point where the polynomial \(h(k^2)\), which is positive for all \(k^2 > 0\) as \(\jac\) is asymptotically stable, is closest to zero. This corresponds to the eigenvalue of \(\jac\) with the largest (negative) real part, which is directly associated with the longest return time. 

{In order to} find the wavenumber $k^2$, within the stability and reactivity region of the matrix $\jac$, that minimizes the distance between the parabola $h(k^2)$ and the $k^2$-axis, let us define
\begin{equation}
    k_{min} :=  \frac{D_u g_v + D_v f_u + \beta g_u \ell(u_0)}{2 D_u D_v}
\end{equation}
that corresponds to the vertex of the parabola $h(k^2)$. Then, as in Section~\ref{sec:react_in}, starting from the conditions for reactivity, it is possible to determine $k^2$ as indicated in Table \ref{tab:k2_cases_clean}.

\begin{table}[h!]
\centering
\small
\renewcommand{\arraystretch}{1.4}
\begin{tabular}{@{} c >{\raggedright\arraybackslash}p{4cm} >{\raggedright\arraybackslash}p{7cm} @{}}
\toprule
\textbf{Case} & \textbf{Conditions} & \textbf{Expression for \( k^2 \)} \\
\midrule

\textbf{1} & \( 0 \leq D_u D_v < \dfrac{\beta^2 \ell^2(u_0)}{4} \) & 
(i) \( \tilde{\Delta} < 0 \): \( k^2 = \max\{0, k_{\min}\} \)\\[0.5em]

& & (ii) \( \tilde{\Delta} > 0 \), \( \tilde{k}_p, \tilde{k}_m > 0 \):\\
& & \hspace*{1em}If \( \tilde{k}_p \leq k_{\min} \leq \tilde{k}_m \):\\
& & \hspace*{2em} \( 
k^2 = 
\begin{cases}
\tilde{k}_p, & \text{if } |k_{\min} - \tilde{k}_p| < |k_{\min} - \tilde{k}_m| \\
\tilde{k}_m, & \text{otherwise}
\end{cases}
\)\\[0.5em]

& & \hspace*{1em}If \( k_{\min} \notin (\tilde{k}_p, \tilde{k}_m) \): \( k^2 = \max\{0, k_{\min}\} \)\\[0.5em]

& & (iii) \( \tilde{k}_p < 0 \), \( \tilde{k}_m > 0 \): \( k^2 = \max\{k_{\min}, \tilde{k}_m\} \)\\[0.5em]

& & (iv) \( \tilde{k}_p, \tilde{k}_m < 0 \): \( k^2 = \max\{0, k_{\min}\} \) \\
\midrule

\textbf{2} & 
\(
D_u D_v \geq \dfrac{\beta^2 \ell^2(u_0)}{4} \geq 0
\)\\
& \( J_{0} \) reactive & 
\( k^2 = \max\left\{0, \min\left\{k_{\min}, \tilde{k}_p\right\}\right\} \) \\
\midrule

\textbf{3} & 
\(
D_u D_v \geq \dfrac{\beta^2 \ell^2(u_0)}{4} > 0
\)\\
& \( J_{0} \) not reactive & 
\( k^2 = \min\left\{ \max\left\{\tilde{k}_m, \tilde{k}_p\right\}, \max\left\{k_{\min}, \tilde{k}_m\right\} \right\} \) \\
\bottomrule
\end{tabular}
\caption{Summary of cases and sub-cases for computing \( k^2 \) based on reactivity and discriminant conditions.}
\label{tab:k2_cases_clean}
\end{table}

 \begin{figure}
     \centering
    \includegraphics[width=0.45\linewidth]{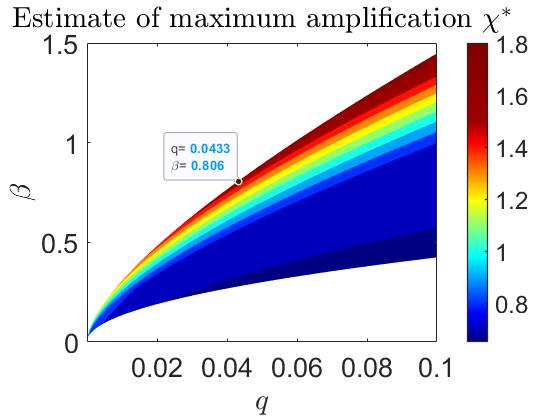}
    \includegraphics[width=0.45\linewidth]{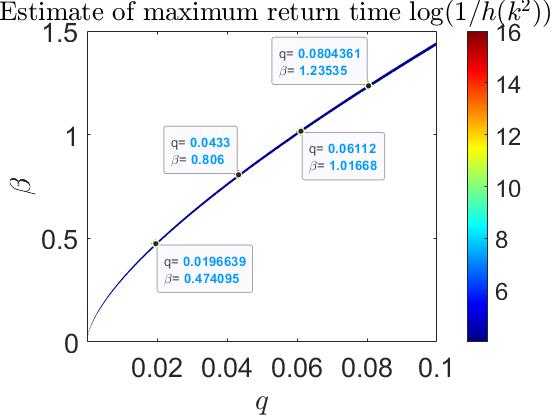}
     \caption{Estimate of the maximum amplification \( \chi^* \) as a function of the parameters \( q \) and \( \beta \). The color scale represents the magnitude of \( \chi^* \), with warmer colors (red) indicating higher amplification values. Amplification \( \chi^* \) increases with both \( q \) and \( \beta \), reaching its maximum in the upper-right region of the plot. The red region, where \( \chi^* > 1.5 \), corresponds to the parameter space where stable reactive patterns are found.  On the right: Values of \( \log\left(1/h(k^2)\right) \) for \( h(k^2) < 0.01835 \) are shown. As \( h(k^2) \) approaches zero, the largest eigenvalues of the Jacobian matrix \( \jac \) approach zero, leading to infinite return times. The threshold \( h(k^2) = 0.01835 \) (equivalent to \( \log\left(1/h(k^2)\right) = 4 \)) identifies the region in the parameter space where additional stable patterns are found.
} 

    \label{fig:maximum_amplification}
 \end{figure}

\begin{figure}
     \centering    
\includegraphics[width=0.45\linewidth]{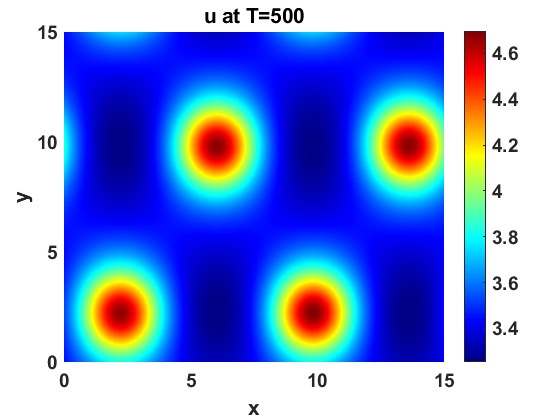}
\includegraphics[width=0.45\linewidth]{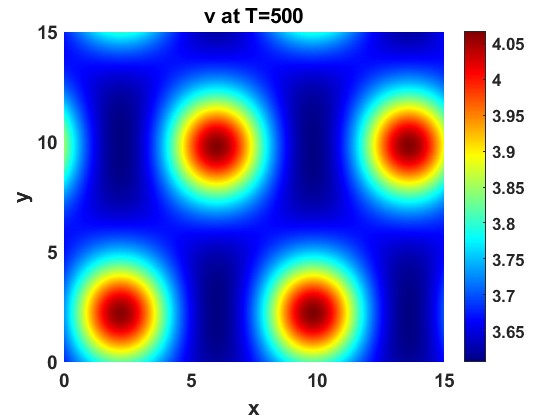}
\includegraphics[width=0.45\linewidth]{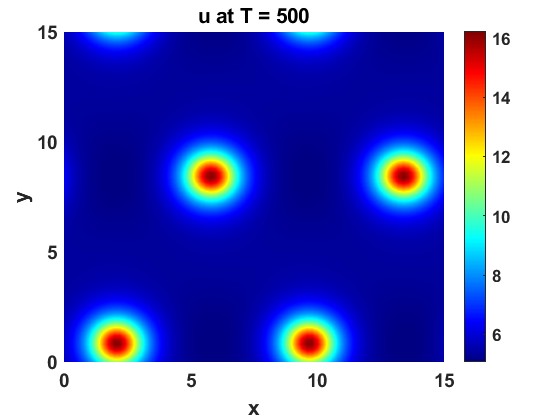}
\includegraphics[width=0.45\linewidth]{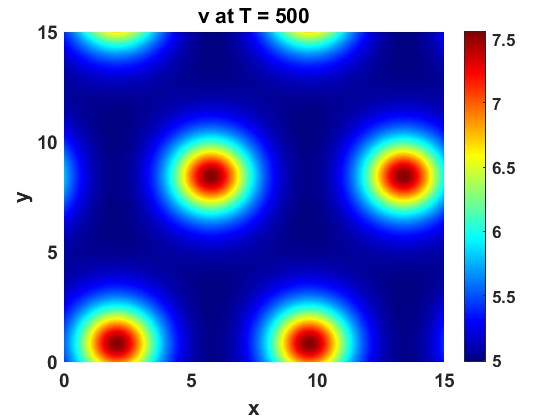}
\includegraphics[width=0.45\linewidth]{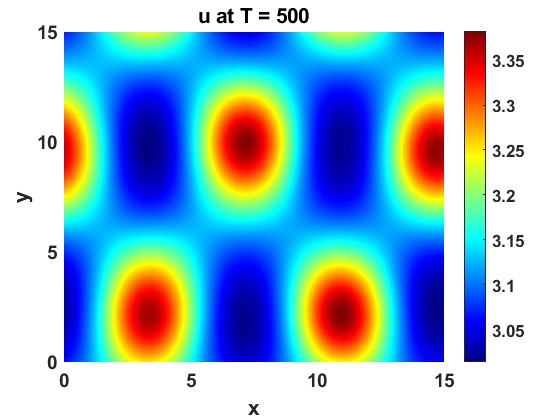}
\includegraphics[width=0.45\linewidth]{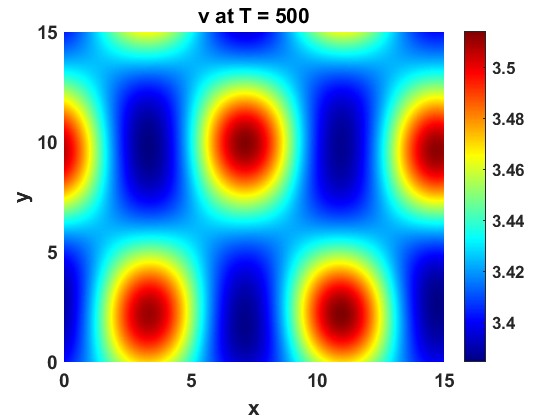} \caption{Pattern of non-normality in the MOMOS model. First row: $q=0.061122$, $\beta=1.01668$. Second row:  $q=0.0196639$, $\beta=0.474095$. Third row:   $q=0.0804361$, $\beta=1.23535$. The other parameters are set to \( D_u = D_v = 0.6 \), \( k_1 = 0.4 \), \( k_2 = 0.6 \), and \( c = 0.8 \).
}
    \label{fig:momos3}
\end{figure}

For these values of \(k^2\), we estimate \(\chi^*\) and visualize the results in Figure~\ref{fig:maximum_amplification}, which shows a heatmap of \(\chi^*\) as a function of the parameters \(q\) and \(\beta\). Amplification \(\chi^*\) increases with both \(q\) and \(\beta\). The points \((q, \beta)\) where we initially observed patterns correspond to values of \(\chi^*\) greater than \(1.5\). We, therefore, focused on identifying patterns of reactivity in the region where amplification values exceed \(1.5\), while staying close to the curve of instability to extend the return time.

 To estimate the return time for the parameter pairs associated with the identified stable patterns, we plot the values of \(\log\left(1/h(k^2)\right)\) as a proxy for the return time. As \(h(k^2)\) approaches zero, the largest eigenvalues of the Jacobian matrix \(\jac\) approach zero, leading to infinite return times. We have identified the threshold \(h(k^2) = 0.01835\) (equivalent to $\log\left(1/h(k^2)\right) = 4$) as the region in the parameter space where the sufficiently long return time allows nonlinear dynamics to intervene, driving the solution toward the basins of attraction of stable non-homogeneous patterns{, although the boundaries between such basins goes far beyond the analysis here}. As it can be seen in Figure~\ref{fig:maximum_amplification} this region is a subset of parameters with maximum amplification $\chi^*>1.5$. Stable reactive patterns have been identified throughout the region $\log\left(1/h(k^2)\right)> 4$ and illustrated in Figure \ref{fig:momos3}.

 In summary, this analysis gives us a broad view of how perturbations will grow, and how long they will persist, as a function of their spatial frequency. As noted, this alone is insufficient for the existence of pattern formation outside of the Turing unstable regime, which requires a stable heterogeneous pattern to be simultaneously stable with the homogeneous equilibrium $P_0$. We next turn to a bifurcation analysis to provide an explanation of where such multistable regions may exist, which coincides with the long-time solutions shown in Figures \ref{fig:momos2} and \ref{fig:momos3}.

\section{Bifurcation analysis}\label{Sec_continuation}

    {Here we will use numerical continuation via pde2path \cite{uecker2014pde2path} to understand the apparent multistability seen in the MOMOS model in the transiently unstable regime (i.e.~the origin of the stable patterned solutions found despite not being in the classical Turing space of asymptotic instability). Motivated from direct numerical explorations, we consider the upper boundary in Figure \ref{fig:momos} between the asymptotic instability regime, and the transient instability regime, which is given by equality of $\beta=\beta_c$ in \eqref{bifurcation_curve}. We will consider 1D and 2D versions of the model to demonstrate the impact of dimensionality on this emergent bistability. We will again make use of the heterogeneity functional given in \eqref{heterogeneity_functional}.

\begin{figure}
        \centering
        \includegraphics[width=0.48\linewidth]{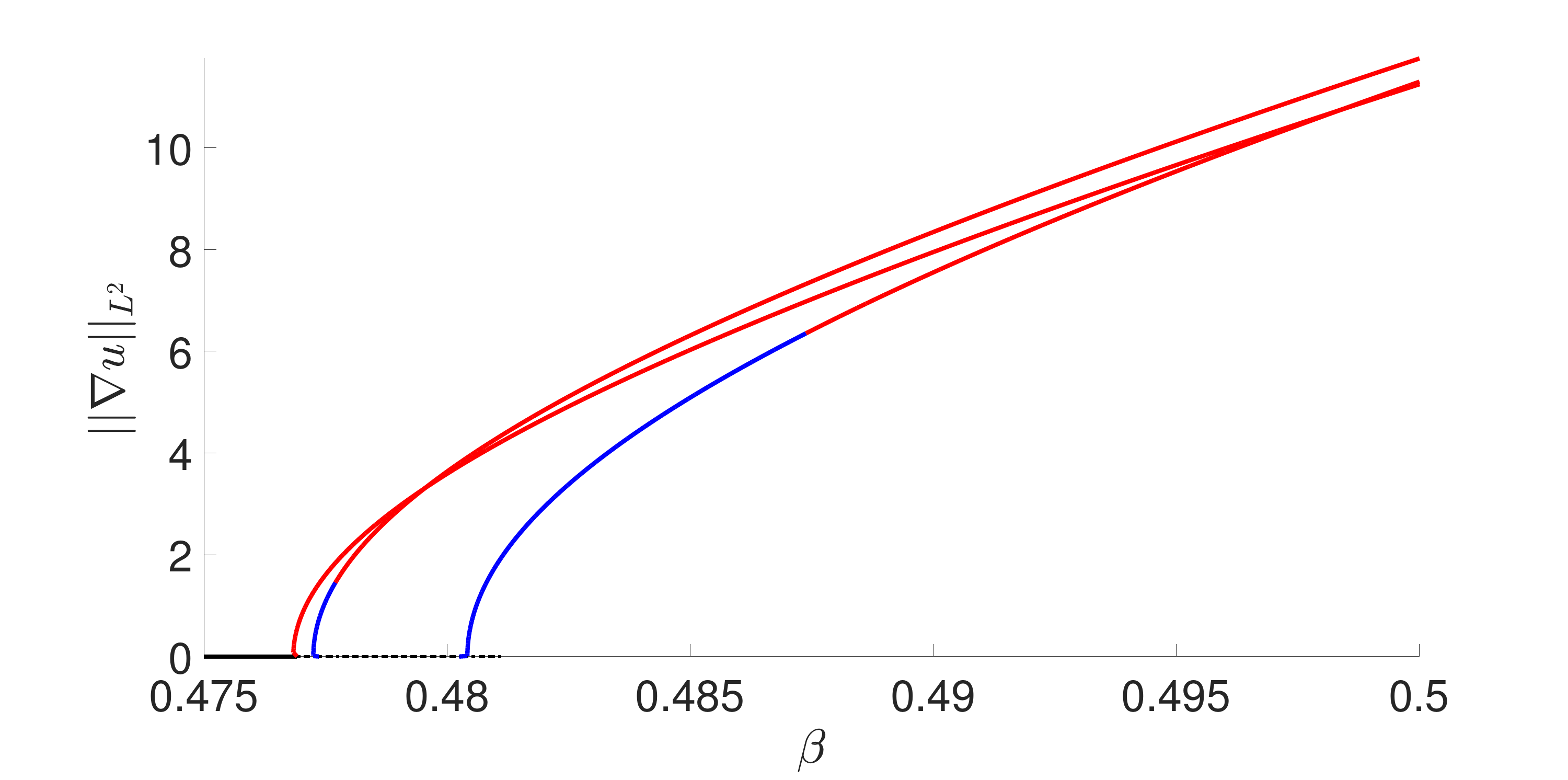}
        \includegraphics[width=0.48\linewidth]{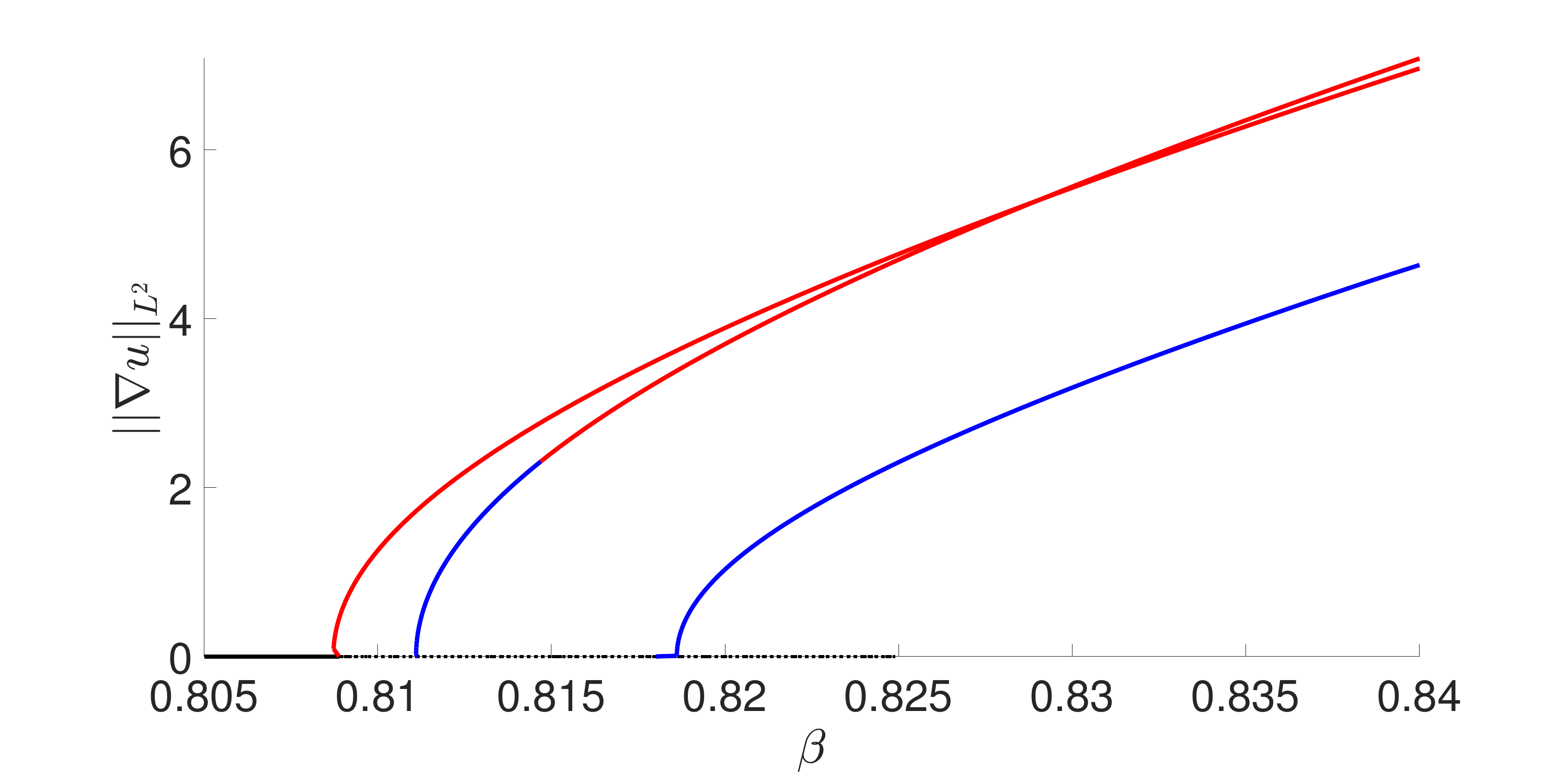}
        
        \hspace{0.4cm} (a) \hspace{5.7cm} (b) \hfill
        
        \caption{{Numerical continuation of steady states of the model \eqref{eq:model} via pde2path, with red curves corresponding to stable solutions, and blue curves corresponding to unstable solutions. We set (a) $q = 0.0196639$ and (b) $q = 0.0433$ and used an interval of length $L = 60$ on a mesh with 2,000 points per side.}}
        \label{fig:1D_continuation}
    \end{figure}

    {We fix all parameters as in Figure \ref{fig:momos}, and vary only $q$ and $\beta$. We remark that the results reported here are for the case of Neumann boundary conditions, although similar structures exist in the periodic case. {We used an extension of the weakly nonlinear analysis code developed by \cite{villar2023computation} for the model in one spatial dimension. We expand $u$ and $v$ in terms of the linearly unstable eigenfunction near the bifurcation curve. Writing $u = u_0 + \varepsilon A(t)\cos(k_c x)$ (and similarly for $v$) where $k_c$ is the critical wavenumber and $\varepsilon^2 \propto \beta-\beta_c$ is the distance from the bifurcation point, then $A(t)$ is the perturbation amplitude which asymptotically satisfies the Stuart-Landau equation:
    \begin{equation}
        \frac{d A}{dt} =  c_1A +c_3A^3 + c_5A^5. 
    \end{equation}
    The cubic coefficient, $c_3$, can be written purely in terms of $q$ along the bifurcation curve (i.e.~by writing $\beta_c$ as a function of $q$ in \eqref{bifurcation_curve}), and is negative for all $q \geq 0$, though the expression is rather long and tedious to sign. We confirm this via numerical continuation in Figure \ref{fig:1D_continuation}, where the supercritical structure of the bifurcation is clear. There are many other branches not shown here, but there was no evidence found that any of the primary or secondary branches folded backwards to give rise to a bistable region. {One can also compare the direct numerical simulations shown in Figure \ref{fig:momos}(b) to these steady state branches to observe how transient instabilities can exploit bistability by tracing out the stable branches.} Direct simulations of the 1D model did not provide any evidence of bistability in the parameter space considered, {consistent with the evidence that the bifurcation is supercritical and no further mechanism exists for generating bistability.}}

    \begin{figure}
        \centering
        \includegraphics[width=0.8\linewidth]{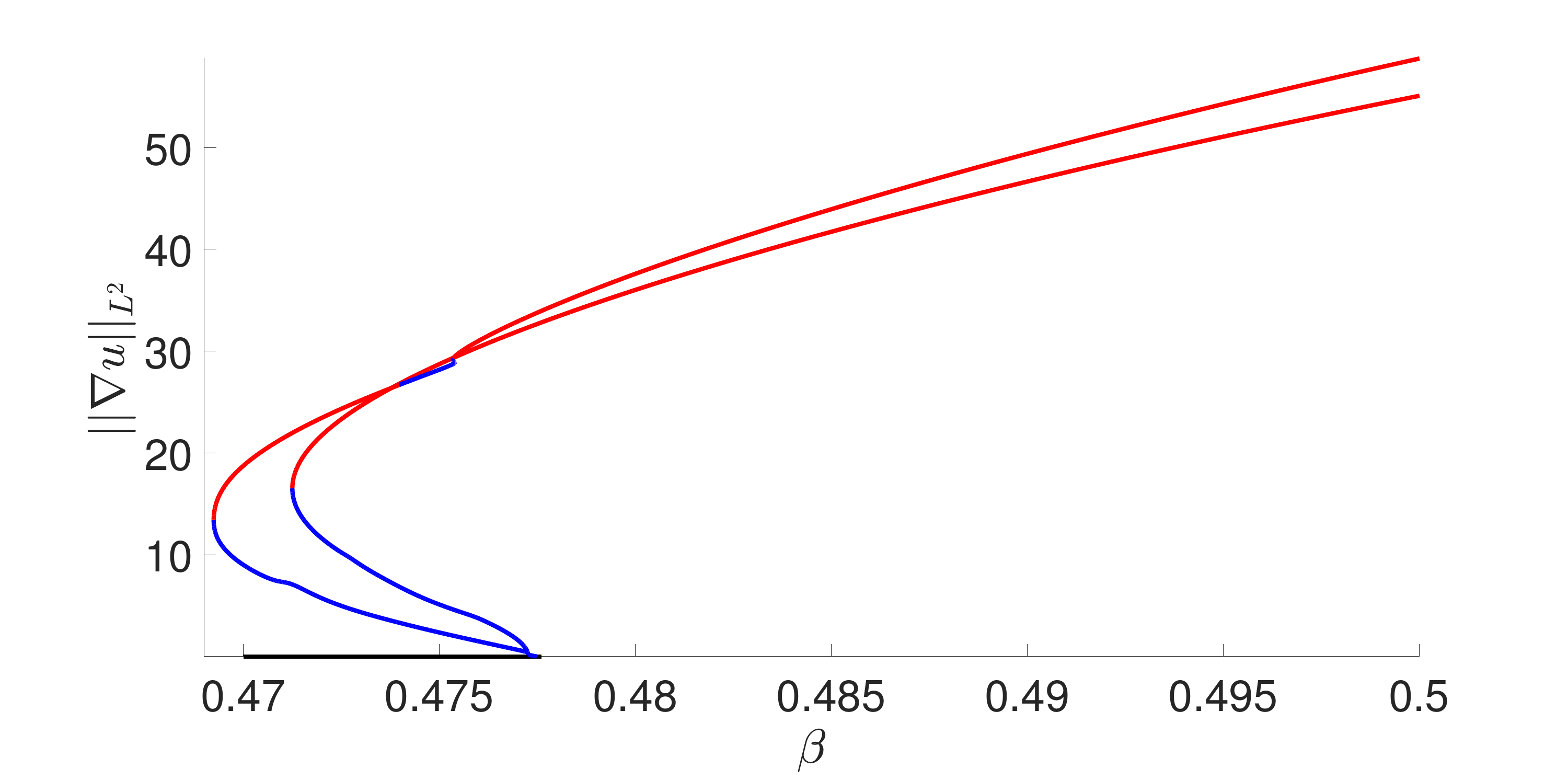}
        \caption{{Numerical continuation of steady states of the model \eqref{eq:model} via pde2path, with red curves corresponding to stable solutions, and blue curves corresponding to unstable solutions. Here $q = 0.0196639$ on a square domain of side length $L = 15$ on a mesh with 90 points per side.}}
        \label{fig:first_continuation}
    \end{figure}
    \begin{figure}
        \centering
        \includegraphics[width=0.48\linewidth]{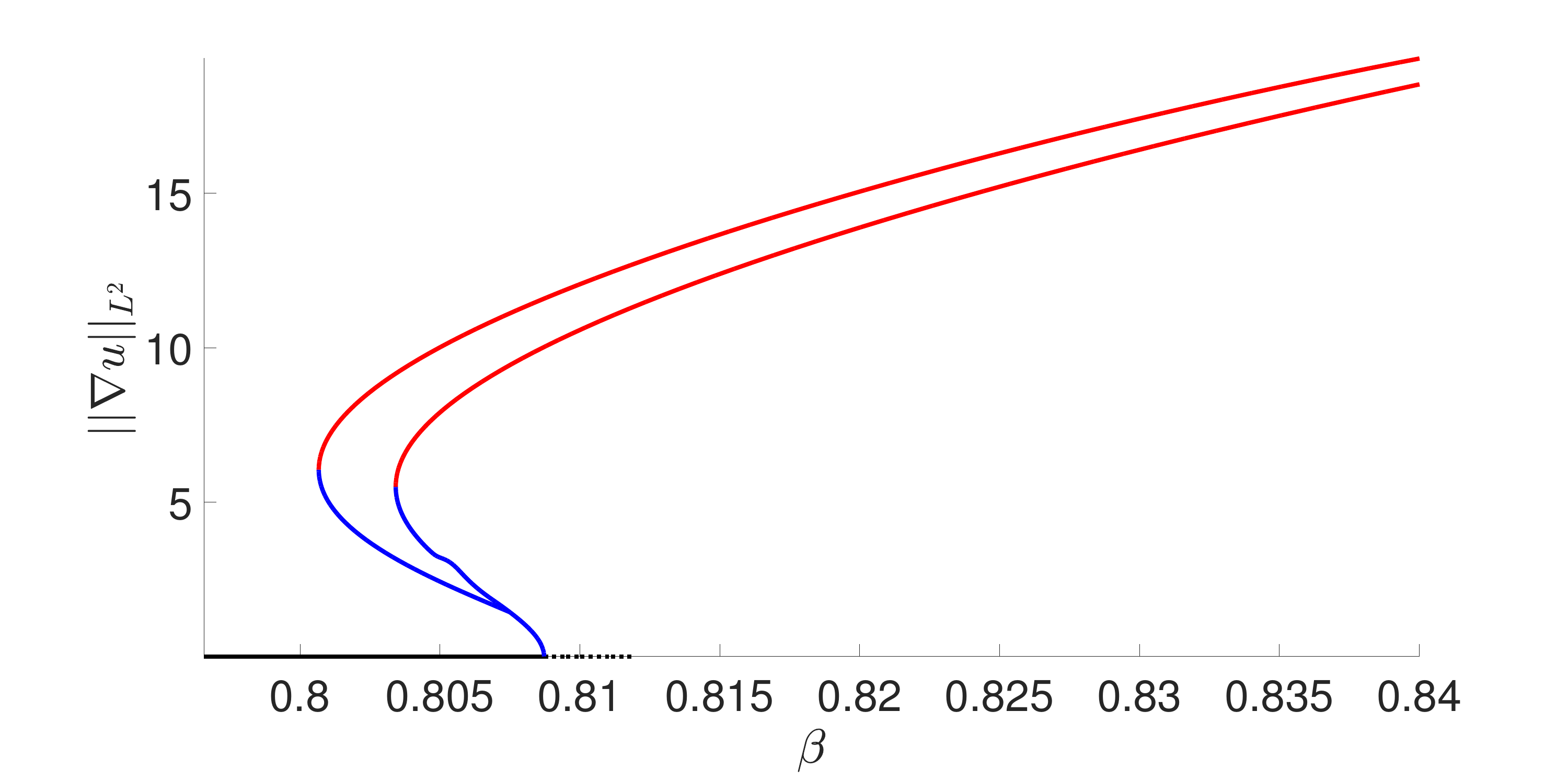}
        \includegraphics[width=0.48\linewidth]{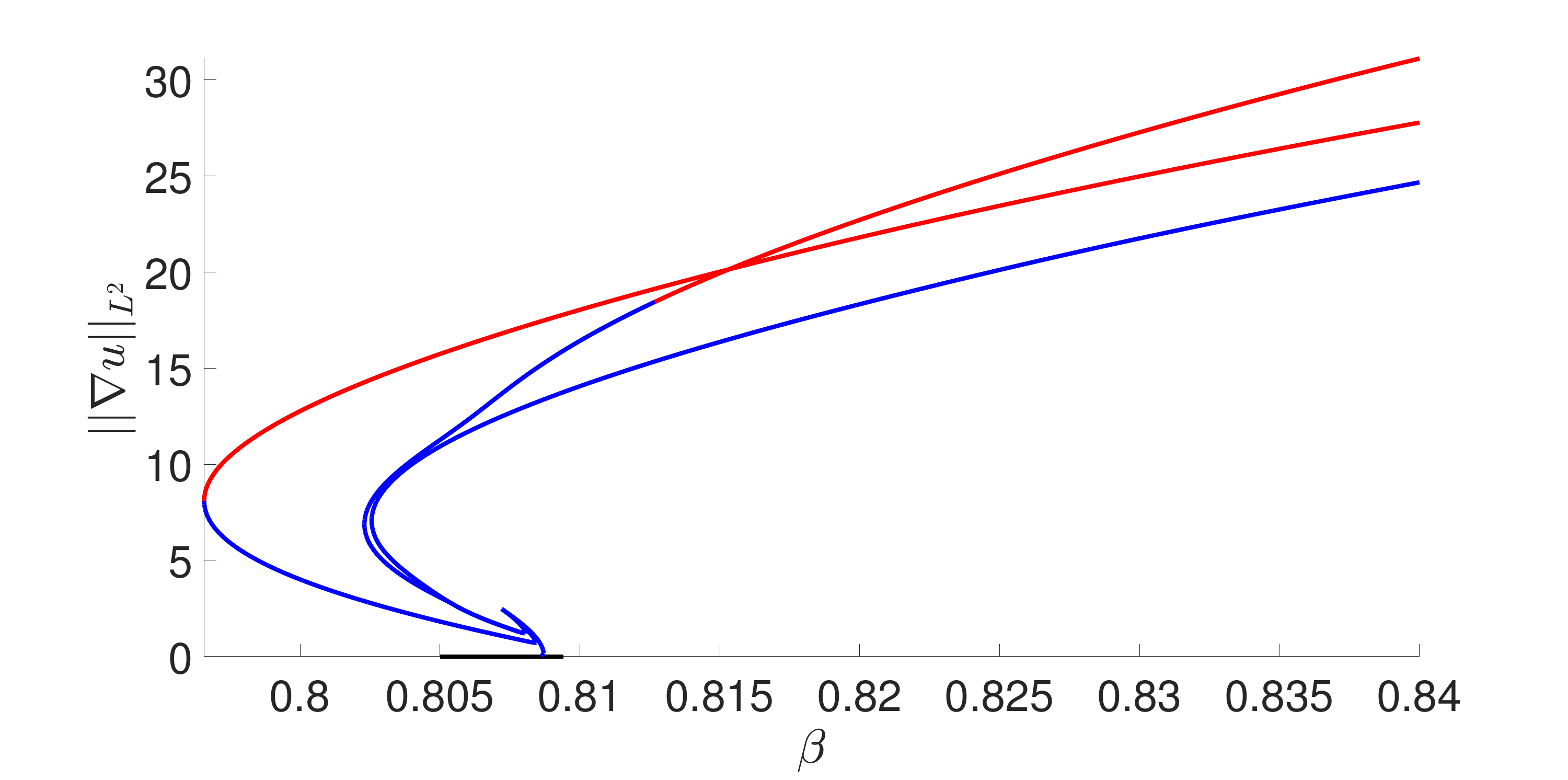}
        
        \hspace{0.4cm} (a) \hspace{5.7cm} (b) \hfill
        
        \caption{{Numerical continuation of steady states of \eqref{eq:model} as in Figure \ref{fig:first_continuation} but with $q = 0.0433$ and (a) $L = 10$, (b) $L=15$.}}
        \label{fig:q0.0433_L10}
    \end{figure}
    
    We next consider the 2D situation via continuation in $\beta$ of equilibrium solutions in Figures \ref{fig:first_continuation} and \ref{fig:q0.0433_L10} for two values of $q$. Due to the 2D geometry and domain length, there are \emph{many} other branches we do not show, and instead we focus on the primary branches emerging from the Turing instability. For these two values of $q$, we see a clear subcritical Turing bifurcation which gives rise to unstable branches that subsequently undergo a fold, leading to a bistable parameter region for smaller values of $\beta$. The size of this bistability region is geometry dependent, as we show in Figure \ref{fig:q0.0433_L10} by comparing domains of different lengths on the same axes. As the domain size is increased further, the location of the fold point (and hence the boundary of the bistable region) does not change appreciably from $L=15$, suggesting that sufficiently large domains will have comparably large bistable regions.
    
    Direct simulations with slowly varying parameters used to quasi-statically continue along stable branches (not shown) also demonstrated a range of different stable structures within these bistable parameter regimes, especially on sufficiently large domains. The key insight here is that the geometry needs to be at least two dimensional, and sufficiently large, to ensure the existence of a bistable region capable of attracting transiently unstable perturbations from the stable homogeneous steady state within the transiently unstable parameter space.}


\section{Discussion, conclusions and future work}\label{Sec_Discussion}

Motivated by experimental observations by \cite{vogel2014submicron}, {as well as the general importance of microbial aggregation and processing of soil organic carbon \cite{zhang2023soil}}, we have analyzed a simplified two-component version of the MOMOS model of bacterial soil carbon dynamics introduced by \cite{hammoudi2018mathematical}. We focused on the chemotaxis-driven pattern formation of this model, and particularly the role of transient instabilities distinct from classical Turing-type routes to pattern formation, {with the goal of characterizing when models might predict microbial aggregation phenomena, i.e.~pattern formation}. 

We demonstrated that the linearized system is non-normal, with initial perturbation amplification driven by its reactivity, reaching a maximum before decaying over the return time. We first revisited the conditions for asymptotic instability and then outlined the conditions under which a general linearized chemotaxis system becomes reactive. These conditions highlight three key scenarios that depend on the balance between chemotaxis and diffusion. Reactivity can emerge when chemotaxis, represented by $|\beta|$, is sufficiently strong to overcome the stabilizing effects of diffusion ($D_u$ and $D_v$). For moderate chemotaxis, reactivity is instead driven by sufficiently large local gradients. In cases where chemotaxis and local gradients are both weak, reactivity arises only through a finely tuned interaction between these parameters and diffusion, governed by additional constraints.

When applied to the MOMOS model, {chemotactic sensitivity ($\beta$), soil organic carbon input ($c$), and the metabolic quotient ($q$) were explicitly given in the conditions for reactivity, suggesting that carbon-relevant parameters could induce transient instabilities and hence plausibly lead to aggregation}. We identified regions of the parameter space $(q, \beta)$ where the system remains both stable and reactive. Within this region, we analyzed the transient dynamics using Klika's indicator for maximum amplification \cite{klika2017significance}, which proved to be a more effective lower bound compared to the Kreiss constant. This approach allowed us to characterize the transient amplification envelope, linking it to the emergence of reactive patterns in regions of asymptotic stability of the homogeneous equilibrium.

A key finding of this analysis was the importance of return time relative to maximum amplification. While  \cite{klika2017significance} concluded that a significant transient growth outside of the asymptotic instability conditions requires $0 < |\lambda_+({k})| < |\lambda_-({k})| \ll 1$, we observed that maximum amplification increases with wavenumber magnitude, and can correspond to values where $|\lambda_-({k})| > 1$. Instead, when the largest (negative) eigenvalue approaches zero (and hence $0 < |\lambda_+({k})| < |\lambda_-({k})| \ll 1$), the return time is maximized. While high maximum amplification can enhance transient deviations, short return times inhibit the system from escaping the basin of attraction of the homogeneous equilibrium. Conversely, even modest maximum amplification can lead to stable pattern formation if the return time is sufficiently long.

By exploiting the determinant of the linearized Jacobian $h(k^2)$ to construct a proxy for return time, we identified critical regions near the instability boundary where sufficiently long return times allow nonlinear dynamics to intervene, facilitating the emergence of stable reactive patterns. Our results further highlighted the role of the chemotaxis effect and nonlinearity governed by the parameter $q$ in driving spatial heterogeneity in soil carbon models. Specifically, while the chemotactic term $\beta$ enhances transient and asymptotic instability, the parameter $q$ has the opposite effect of stabilizing the dynamics. 

{Our study also exemplifies some key differences between chemotaxis models and reaction-diffusion systems. As is well-known, Equation (\ref{eq:suff_conditions_cd}) implies that strong chemotactic aggregation alone (given by large $\beta$) is enough to induce Turing instability somewhat independently of the rest of the details of the system. This is in contrast to reaction-diffusion models, where Turing instability requires rather constraining conditions on the parameters (e.g.~Jacobian matrices with particular sign structures).  Along these lines, Case I of our conditions for reactivity highlights an analogous effect; the size of the parameter space region where reactivity is possible is directly proportional to the chemotactic strength $\beta$, rather than requiring a more fine-tuned balance of the system parameters. We expect that an analogous generalization holds for more general cross-diffusion systems, as these often have many distinct ways of inducing pattern-forming instabilities \cite{villar2024designing}). Additionally, the chemotaxis-driven instability here does not require any reactivity of the spatially homogeneous Jacobian, which was suggested by \cite{NEUBERT2002} to be a rather general principle underlying wide classes of systems exhibiting pattern-forming instabilities.}

{For the specific model and parameters investigated here, the existence of patterns of reactivity only occurred near the boundary of a subcritical bifurcation, where the well-known bistable region gave rise to their existence. Compared with the results of \cite{klika2017significance}, while the transient instability region in the case of chemotaxis can be much larger than the reaction-diffusion examples studied there, the existence of stable patterns is constrained to this bistable region, and hence to the boundary of the classical Turing space.  There are many systems where pattern formation can occur in the absence of a Turing bifurcation (e.g.~\cite{brauns2020phase, al2022stationary}), and hence where one could expect stable patterned states to exist and be reachable from a transient instability of a reactive equilibrium, but we leave such a conjecture for future work.}

{It would also be interesting to explore whether patterns of reactivity are likely in other models, near subcritical Turing bifurcations or otherwise. \cite{krause2024turing} suggested that multistability may be common in large biological systems such as gene regulatory networks underlying vertebrate pattern formation, and such multistability can lead to the loss of pattern formation after a Turing instability. In principle, transient spatial instability could have the opposite effect, leading to pattern formation without any Turing instability. In general, many analyses of reaction-(cross)-diffusion models neglect the possibility of transient instabilities and multistability of patterned states, solely focusing on understanding asymptotic stability through, e.g., classical Turing space pictures. Such classical tools are both conceptually simpler, and easier for large-scale systematic analyses (e.g.~see \cite{marcon2016high} for an example). Our work instead provides another example where these effects can matter, hence suggesting the value in more nuanced understandings of  dynamics in these systems.}

From an applied perspective, transient patterns arising from reactivity may have significant implications for understanding hot-spot formation in soil ecosystems, as experimentally observed in \cite{vogel2014submicron}. These patterns could influence microbial activity, carbon and nitrogen sequestration, and soil health by creating localized zones of intense organic matter processing. {The role of microbial aggregation, and its interaction with soil constituents such as carbon and nitrogen, may play important roles in the overall heterogeneous structure of soil and its impact on soil organic carbon accumulation \cite{zhang2023soil}. Standard analyses of such processes focus on systematic linear stability, but our results suggest that more nuanced dynamics, such as transient amplification, can modify the parameter space under which spontaneous microbial aggregation is observed. In the context of carbon sequestration, one would ideally want to increase both $c$ and $q$ to maximize the amount of accumulated carbon, albeit these have opposing effects on both pattern formation given in \eqref{bifurcation_curve} and reactivity conditions given by, e.g., Case 1 above.}

Future studies should aim to integrate the dynamics of chemotaxis-driven patterns with experimental data on organic matter binding and clustering. Additionally, incorporating the effects of external environmental factors, such as moisture and temperature gradients, could enhance the predictive power of such models. 
 Furthermore, delving deeper into the pattern formation of the MOMOS model for soil carbon dynamics can be crucial in the emerging context of the idea of the `Circular Economy' \cite{Yin_2023}, as it can offer a more profound understanding of the behavior and distribution of carbon in soils over time. This knowledge can promote regenerative agriculture, the efficient use of organic waste, and agricultural practice optimization, all of which are consistent with the Circular Economy's guiding ideals of reducing waste, increasing resource efficiency, and promoting environmental sustainability \cite{Wang_2024,Nwaogu_2025}.  Future work should also focus on extending these analyses to incorporate additional biological complexities, such as multi-species interactions, stochastic perturbations and spatial heterogeneities, with a particular attention to the implementation of suitable control approaches \cite{Mehdaoui_2024,Mehdaoui_2025}. 

 The bistable regions observed here were only found near the Turing bifurcation boundary, and only in two dimensions where this bifurcation appeared to be subcritical for all parameters considered. However, more complex cross-diffusion models may exhibit other mechanisms of multistability where such transient amplification can play a role in widening the regimes of pattern formation. Moreover, there is growing evidence that over-reliance on classical linear stability may fail to work for large classes of systems, such as those where transient Turing instabilities generically lead to alternative homogeneous equilibria \cite{krause2024turing}. This suggests the need for better tools to understand increasingly-realistic models, especially in the spatial setting where the number of equilibria (and hence the overall complexity of attractors) can be prohibitively large. Alternative approaches to understanding more global aspects of such systems (for instance, those reviewed by \cite{krakovska2024resilience}) are likely fruitful paths to try and understand the emergence and resilience of complex states, such as patterns. 
 
\section*{Acknowledgements}

A.M., F.D. and  C.M. research activity is funded by the National Recovery and Resilience Plan (NRRP), Mission 4 Component 2 Investment 1.4 - Call for tender No.
3138 of 16 December 2021, rectified by Decree n.3175 of 18 December 2021 of Italian Ministry of University and Research funded by the European Union – NextGenerationEU; Award Number: Project code CN 00000033, Concession Decree No. 1034 of 17 June 2022 adopted by the Italian Ministry of University and Research, CUP B83C22002930006, Project title \lq \lq National Biodiversity Future Centre''. \\
F.D. and A.M. gratefully acknowledge the support of INdAM-GNCS 2025 Project \lq \lq ROOTS: Research on sOil Organic carbon dynamics, numerical and data-driven meThods for Spatial pattern'', grant number CUP E53C24001950001.\\
D.L. research has been developed within the Project \lq \lq P2022PSMT7'' CUP H53D23008940001 funded by EU in NextGenerationEU plan through the Italian "Bando Prin 2022 - D.D. 1409 del 14-09-2022" by MUR.\\
F.D. and A.M. are members of the INdAM research group GNCS; D.L. is member of the INdAM research group GNFM. F.D., C.M. and A.M. would like to thank Mr. Cosimo Grippa for his valuable technical support. 

\bibliographystyle{unsrt} 
\bibliography{biblio}

\end{document}